\documentclass[11pt,reqno]{amsart} 
\usepackage[a4paper,left=24mm,right=24mm,top=32mm,bottom=30mm]{geometry} 

\usepackage[latin1]{inputenc}
\usepackage[ngerman,english]{babel}

\usepackage[T1]{fontenc}
\usepackage{lmodern}
\usepackage{a4wide}
\usepackage{color}
\usepackage{enumerate}
\usepackage{amsmath}%
\usepackage{amsfonts}%
\usepackage{amssymb}
\usepackage{amsthm}
\usepackage{amsmath,amsxtra,amssymb,latexsym, amscd,amsthm}
\usepackage{bbm}
\usepackage{empheq}
\usepackage{graphicx}
\usepackage{MnSymbol}
\usepackage{wrapfig}
\usepackage[percent]{overpic}
\usepackage{setspace}
\usepackage{eurosym}
\usepackage{cite}
\usepackage{faktor}
\usepackage{xfrac}
\usepackage{esint}

\usepackage{subfigure}

\usepackage{enumitem}
\usepackage{xcolor}

\newtheorem{theorem}{Theorem}[section]
\newtheorem{proposition}[theorem]{Proposition}
\theoremstyle{definition}

\newtheorem{condition}[theorem]{Condition}

\newtheorem{lemma}[theorem]{Lemma}

\newtheorem{remark}[theorem]{Remark}

\numberwithin{equation}{section}

\providecommand{\norm}[1]{\left\lVert #1 \right\rVert}
\providecommand{\abs}[1]{\left\lvert #1 \right\rvert}

\newcommand{\F}{\mathcal F}                
\newcommand{\eps}{\varepsilon}             
\newcommand{\pd}{\partial}                 
\newcommand{\R}{\mathbb{R}}                
\newcommand{\N}{\mathbb{N}}                
\newcommand{\SL}{\Delta_{\Gamma}}          
\newcommand{\SG}{\nabla_{\Gamma}}          
\newcommand{\G}{\Gamma}
\newcommand{\spa}{\operatorname{span}}

\newcommand{\tr}{\operatorname{tr}}

\newcommand{\phit}{\tilde{\varphi}^N}
\newcommand{\ut}{\tilde{u}^N}
\newcommand{\vt}{\tilde{v}^N}
\newcommand{\mut}{\tilde{\mu}^N}
\newcommand{\thetat}{\tilde{\theta}^N}
\newcommand{\solspace}{\mathcal{W}}

\newcommand{\vectthree}[3]{\left(\begin{matrix}#1\\  #2\\ #3 \end{matrix}\right)}

\newcommand{\dprodH}[3]{\left\langle #2, #3 \right\rangle_{H^{-1}(#1),H^1(#1)}}

\title[On a Model for Phase Separation on Biological Membranes]{On a Model for Phase Separation on Biological Membranes and its Relation to the Ohta-Kawasaki Equation}
\author[H. Abels]{Helmut Abels}
\address{Helmut Abels, Fakult\"at f\"ur Mathematik, Universit\"at Regensburg, 93040 Regensburg, Germany, e-mail: helmut.abels@mathematik.uni-regensburg.de}
\author[J. Kampmann]{Johannes Kampmann}
\address{Johannes Kampmann, Fakult\"at f\"ur Mathematik, Universit\"at Regensburg, 93040 Regensburg, Germany, e-mail: johannes.kampmann@mathematik.uni-regensburg.de}

\begin{document}

\begin{abstract}
	We provide a detailed mathematical analysis of a model for phase separation on biological membranes which was recently proposed by Garcke, Rätz, Röger and the second author. The model is an extended Cahn-Hilliard equation which contains additional
	terms to account for the active transport processes.
	We prove results on the existence and regularity of solutions, their long-time behaviour, and	on the existence of stationary solutions. Moreover, we investigate two different asymptotic regimes. We study the case of large cytosolic
	diffusion and investigate the effect of an infinitely large affinity between membrane components. The first case leads to the reduction of coupled bulk-surface equations in the model	to a system of surface equations with non-local contributions. Subsequently, we recover a variant of the well-known Ohta-Kawasaki equation as the limit for infinitely large affinity between membrane components.
\end{abstract}

\keywords{partial differential equations on surfaces, phase separation, Cahn-Hilliard equation, Ohta-Kawasaki equation, reaction-diffusion systems}
\subjclass[2010]{35K52,35Q92,92C37}

\maketitle
\section{Introduction}

Because of their importance for many physical or biological systems, phase separation processes have been thoroughly studied and depending on the concrete application, different mathematical models have been developed to model such behaviour. In material science, common models are the Allen-Cahn \cite{AC} and Cahn-Hilliard equation \cite{CA,CH}, which are based on the Ginzburg-Landau energy, or the Ohta-Kawasaki equation \cite{OK}, which is derived from an additional non-local contribution to the Ginzburg-Landau energy. There are many articles discussing the derivation and properties of these models, e.g. \cite{LQC, CH1d, CHL, BHC} in the case of the Allen-Cahn equation, \cite{NC} and \cite{EZ,T, NST1, NST2} for the Cahn-Hilliard equation, and \cite{ChR, ONIY, RW} in the case of the Ohta-Kawasaki equations.

In contrast to the examples from material science above, models for microphase separation on biological membranes are relatively new. We refer the reader to the overview articles \cite{FS, RL, FSH} for a comprehensive introduction to several phenomenological models. 

One main aspect of phase separation in this context is the emergence of microdomains with a length-scale below the system size. It has been argued that cell membranes are affected by active cellular processes which contribute to this behaviour and effectively keep the phase separation process from reaching its equilibrium \cite{FSH,RL,GSR,Fo}. 

In \cite{GKRR} Garcke et al. derived a model for phase separation on biological membranes from thermodynamics which includes such non-equilibrium contributions.

To the knowledge of the authors, recent contributions emphasize the derivation of models and qualitative behaviour or simulations while neglecting other aspects of a detailed mathematical analysis. For the lipid raft model proposed in \cite{GKRR}, we will carry out such an analysis of its mathematical properties in this paper. In particular, we prove existence and regularity for solutions and give a result establishing a connection between the model in \cite{GKRR} and the well-known Ohta-Kawasaki equations. 

Biological membranes generally consist of bilayers of phospholipid molecules, but can also include other molecules such as cholesterols or proteins. Phospholipids are molecules composed of two hydrophobic fatty acids which are linked through a hydrophilic phosphate group. They arrange themselves in a bilayer, i.e. in two layers of lipid molecules with the hydrophobic tails pointing towards each other.
In eukaryotic cells, such a bilayer cell membrane encloses the cytosol, the cellular fluid inside the cell. 

Cell membranes are highly heterogeneous, containing lipids with either saturated or unsaturated tails as well as cholesterols, proteins and other molecules. The lateral organisation of these different components is important for the functioning of the cell, contributing to protein trafficking, endocytosis, and signalling \cite{FSH2, RL}.  

A lot of attention in this context is given to the emergence of so-called lipid rafts. These rafts are intermediate sized domains ($10-200$  nm), characterized as regions consisting mainly of saturated lipid molecules enriched with cholesterols \cite{PIK}. We refer the reader to the overview \cite{FS} and the list of references therein for a discussion of the experimental evidence for their existence.

From a mathematical point of view, these lipid rafts have the striking feature that they do not merge in such a way that the area of the phase boundary is minimized. Instead, they develop into several finite-size domains. This behaviour differs from other phase separation processes and models for lipid raft formation need to capture this behaviour. 

Due to their structure with a semirigid tail, cholesterol molecules have a strong affinity for saturated lipids, and regions with a high concentration of saturated lipids, which are enriched in cholesterol are much more ordered than regions in which cholesterol is absent \cite{RPVK}. As such, active transport processes of membrane components like cholesterol and lipids must be taken into account as non-equilibrium contributions when discussing lipid raft formation from a thermodynamical point of view. In particular, it has been observed that the formation of lipid rafts is linked to the presence of cholesterols in the membrane \cite{LR}.

Based on this assumption, several theoretical models for the formation of lipid rafts have been proposed, see \cite{FS, RL, FSH}.

\section{Problem Statement and Main Results}
In this section, we give the exact statement of the lipid raft model considered in this paper and present our main results. The results are part of the second author's PhD thesis~\cite{Diss}.

A short introduction to the lipid raft model is sufficient for the mathematical analysis in the present paper. A comprehensive discussion and the full derivation from thermodynamics is given in \cite{GKRR}.

Let $B \subset \R^3$ be a bounded domain with smooth boundary $\G:=\partial B.$ The domain $B$ and the surface $\G$ represent the cell and its outer membrane respectively. The basic quantities in the model are the rescaled relative concentration $\varphi$ of saturated lipids in the membrane, the relative concentration $v$ of membrane-bound cholesterol and the relative concentration $u$ of cytosolic cholesterol. We normalize $\varphi$ such that $\varphi = 1$ represents the pure saturated lipid phase and $\varphi=-1$ within the pure unsaturated lipid phase. Moreover, $v=1$ and $u=1$ correspond to maximal saturation for the cholesterol concentrations.

Let now for $\eps,\delta >0$
\begin{equation}
\label{eq:energy}
\F(v, \varphi) = \int_\Gamma\frac\eps2|\nabla\varphi|^2 +
\eps^{-1}W(\varphi) + \frac1{2\delta}(2v - 1 - \varphi)^2 \ d\mathcal{H}^2,
\end{equation}
with the double-well potential $W(s)=(1-s^2)^2.$ The functional $\F$ consists of two parts. The first part $\int_\Gamma\frac\eps2|\nabla\varphi|^2 +
\eps^{-1}W(\varphi)\ d\mathcal{H}^2$ is a classical Ginzburg-Landau energy, modeling the phase separation between the two lipid phases. The second part $\frac1{2\delta}\int_{\G} (2v - 1 - \varphi)^2\ d\mathcal{H}^2$ accounts for the affinity between saturated lipid molecules and membrane-bound cholesterol.

We now assume that the evolution of the membrane quantities is driven by chemical potentials derived from the functional $\F.$ Namely, we introduce
\begin{align*}
\mu &:= \frac{\delta \F}{\delta \varphi} 
=  - \eps \SL \varphi + \eps^{-1}W'(\varphi) - \delta^{-1}(2v - 1 - \varphi),\\
\theta &:= \frac{\delta \F}{\delta v} = \frac2\delta(2v - 1 - \varphi),
\end{align*}
and say that $\F$ is the surface free energy functional of the model.

We then consider the following bulk--surface system consisting of a surface Cahn--Hilliard equation coupled by an exchange term $q$ to a bulk--diffusion equation,
\begin{alignat}{2}
\label{eq:diffU}
\partial_t u &= D \Delta u &\qquad \text{in } B \times (0,T]&,\\
\label{eq:flux}
- D \nabla u \cdot \nu & = q &  
\text{on } \Gamma \times (0,T]&,\\
\label{eq:CH1}
\pd_t \varphi &= \SL \mu
&\text{on } \Gamma \times (0,T]&,\\
\label{eq:CH2}
\mu &= - \eps \SL \varphi + \eps^{-1}W'(\varphi) - \delta^{-1}(2v - 1 - \varphi)
&\text{on } \Gamma \times (0,T]&,\\
\label{eq:v}
\pd_t v &= \SL \theta + q = \frac4\delta\SL v - \frac2\delta\SL
\varphi + q
&\text{on } \Gamma \times (0,T]&,\\
\theta &=  \frac2\delta(2v - 1 -
\varphi)&\text{on } \Gamma \times (0,T]& \label{eq:theta}
\end{alignat}
with initial conditions for $u$, $\varphi$ and $v$. Here we denote by $\nu$ the outer unit normal vector of $B$ on $\G$, $D>0$ is the diffusion coefficient, and $T\in (0,\infty)$ is arbitrary. The exchange term $q$ will be specified later.

A few comments on the basic ideas included in these equations are in order. From a thermodynamical viewpoint, \eqref{eq:CH1} and \eqref{eq:v} are mass balance equations for the surface quantities. Equations \eqref{eq:diffU} and \eqref{eq:flux} model the evolution of the cytosolic cholesterol by a simple diffusion equation with diffusion coefficient $D>0.$ The important part is the inclusion of Neumann boundary conditions for the cytosolic diffusion. Depending on the characterization of the exchange term $q$, the cholesterol flux from the cytosol $B$ onto the membrane $\G$ appears as a source term for the evolution of the membrane-bound cholesterol $v$ in equation \eqref{eq:v}. Equation \eqref{eq:v} also includes a cross-diffusion, which stems from the cholesterol-lipid affinity in the surface energy $\F$. Finally, equations \eqref{eq:CH1} and \eqref{eq:CH2} constitute Cahn-Hilliard dynamics for the lipid concentration and allow for a contribution from the cholesterol evolution via the last term. We note that the parameter $\delta$ effectively controls how much the preferred binding between saturated lipids and cholesterols influences the system.

\begin{remark}\label{rem:thermo_consistent}
	The discussion in \cite{GKRR} shows that the model is thermodynamically consistent for arbitrary constitutive choices for the exchange term $q$. The authors derive the model from mass balance equations for the relative lipid concentration $\varphi$ and the cholesterol concentration $v$ on the surface $\G$ as well as the mass balance for the cytosolic cholesterol concentration $u.$ In both the cholesterol mass balance equation on $\G$ and the cholesterol mass balance equation in $B,$ the exchange term $q$ is treated as an external source. 
\end{remark}
\begin{remark}\label{rem:mass_cons}
	Equation \eqref{eq:CH1} implies that the total mass of surface lipids $\int_{\G} \varphi \ d\mathcal{H}^2$ is constant in time. Similarly, equations \eqref{eq:diffU}, \eqref{eq:flux}, and \eqref{eq:v} yield that the combined total mass of surface and cytosolic cholesterol $\int_B u \ dx + \int_{\G} v \ d\mathcal{H}^2$ is conserved. We will always denote the total lipid and cholesterol mass by $m$ and $M$ respectively, i.e. for all times
	\begin{align}\label{eq:mass_cons}
	\int_{\G} \varphi \ d\mathcal{H}^2 = m, \quad\quad
	\int_B u \ dx + \int_{\G} v \ d\mathcal{H}^2 = M.
	\end{align}
	Moreover, the time derivative of the sum of surface and bulk energy fulfills
	\begin{align}\label{eq:energy_thermo}
	\frac{d}{dt} \left( \mathcal{F}(v(\cdot,t),\varphi(\cdot,t)) + \frac{1}{2}\int_B u(\cdot,t)^2 \right) \leq \int_{\G} q(\theta-u) \ d\mathcal{H}^2.  
	\end{align}
	In particular, whether $\mathcal{F}(v(\cdot,t),\varphi(\cdot,t)) + \frac{1}{2}\int_B u(\cdot,t)^2$ is decreasing in time or not depends on the choice of the exchange term $q.$ It is possible to prove that the energy stays bounded on finite time intervals under suitable growth assumptions on $q.$ This will be the key ingredient for our following existence proof. 
\end{remark}
Define
\[ \solspace := \solspace_B \times \solspace_\Gamma^1 \times \solspace_\Gamma^1\times \solspace_\Gamma^2\times \solspace_\Gamma^2, \]
where
\begin{align*}
\solspace_B  &:= L^2\left(0,T;H^1(B)\right) \cap  H^1\left(0,T;\left(H^{1}(B)\right)'\right),\ \\ 
\solspace_\Gamma^1  &:= L^2(0,T;H^1(\Gamma))\cap H^1(0,T;H^{-1}(\Gamma)),  \text{ and }
\solspace_\Gamma^2  := L^2(0,T;H^1(\Gamma)).
\end{align*}
\begin{theorem}[Existence of Weak Solutions]\label{thm:existence}~\\
	Let $T\in(0,\infty).$ Let $\varphi_0 \in H^1(\Gamma), v_0\in L^2(\Gamma)$ and $u_0 \in L^2(B)$. 
	Moreover, assume that the exchange term $q: \R^2 \rightarrow \R$ is continuous and fulfils
	\begin{equation}\label{eq:growth_cond_q}
	\abs{q(u,v)}\leq C(1+\abs{u}+\abs{v}) \qquad \forall \ u,v \in \R
      \end{equation}
      for some $C>0$.
	Then there exist functions $(u,\varphi,v,\mu,\theta) \in \solspace$
	which are a weak solution to problem \eqref{eq:diffU}--\eqref{eq:theta}, i.e. they fulfil for all $\xi \in L^2(0,T;H^1(B))$ and $\eta \in L^2(0,T;H^1(\Gamma))$ the equations
	\begin{align}
	\int_0^T \left\langle \partial_t u, \xi \right\rangle_{\left(H^1(B)\right)',H^1(B)} &= -\int_0^T\int_B \nabla u \cdot \nabla \xi - \int_0^T\int_{\Gamma} q(u,v)\xi,\label{eq:weak_form_diffU} \\
	\int_0^T \dprodH{\G}{\partial_t \varphi}{\eta} &= -\int_0^T\int_{\Gamma} \SG \mu \cdot \SG \eta, \label{eq:weak_form_time_der_phi} \\
	\int_0^T\int_{\Gamma} \mu \eta &= - \int_0^T\int_{\Gamma} \left[\varepsilon \SG \varphi \cdot \SG \eta + \frac{1}{\varepsilon} W'(\varphi)\eta -\frac{1}{\delta}\left( 2v - 1 -\varphi \right)\eta \right],\\
	\int_0^T \dprodH{\G}{\partial_t v}{\eta} &= -\int_0^T\int_{\Gamma} \SG \theta \cdot \SG \eta +\int_0^T\int_{\Gamma} q(u,v)\eta, \\
	\theta  &= \frac{2}{\delta} \left( 2v - 1 - \varphi\right) \text{ a.e. on } \G\times(0,T).\label{eq:weak_form_theta}
	\end{align}
	 The initial values are attained in $L^2(B)$ and $L^2(\G)$ respectively. Moreover,  \begin{align}
	\F(v(\cdot,t),\varphi(\cdot,t)) + \frac{1}{2}\int_B u(\cdot,t)^2
	+ \int_0^t\int_B \frac{D}{2}|\nabla u|^2 + \int_0^t \int_{\Gamma} \left( \abs{\nabla_\Gamma \mu(\cdot,t)}^2 + \abs{\nabla_\Gamma \theta(\cdot,t)}^2 \right) 
	\, \nonumber \\ \leq\, C(\eps,\delta,\Lambda,T,D_0,v_0,\varphi_0,u_0). \label{eq:energy_est_lin_growth}
	\end{align} holds for any $D\geq D_0>0$ and all $t \in (0,T].$
\end{theorem}
The proof of Theorem \ref{thm:existence} will be given in Section \ref{sec:proof_ex} below.
\begin{remark}[Uniqueness]
  If additionally to the assumptions above $q$ is globally Lipschitz continuous, one can prove in a straight forward manner that weak solutions are unique. To this end one tests \eqref{eq:diffU} for the difference $u=u_1-u_2$ of two weak solutions $(u_j, \varphi_j, v_j, \mu_j,\theta_j)$, $j=1,2$, with $u$ in $L^2(B)$, tests \eqref{eq:CH1} with $\Delta^{-1}\varphi=\Delta^{-1}(\varphi_1-\varphi_2)$ in $L^2(\Gamma)$, \eqref{eq:CH2} with $\varphi$ and \eqref{eq:v} with $v=v_1-v_2$. Moreover, one uses compactness of the trace operator $\tr_{\Gamma}\colon H^1(B)\to L^2(\Gamma)$ together with Ehrling's Lemma and of course Gronwall's inequality. Uniqueness also holds true if $q$ is only locally Lipschitz continuous and the weak solutions are essentially bounded. The existence of bounded weak solutions follows from the next result under suitable assumptions.
\end{remark}

Once we know that solutions exist on any finite time interval, we can address their regularity. Provided that the exchange term $q$ fulfils additional growth assumptions, we obtain higher regularity for solutions to the lipid raft model.
\begin{theorem}[Higher regularity]\label{thm:higher_reg}~\\
	Let $u_0, v_0$ and $\varphi_0$ be as in Theorem \ref{thm:existence}. Assume that $q \in C^1(\R^2)$ such that
	\begin{align}
	\abs{q(u,v)} &\leq C(1+\abs{u}+\abs{v}) \nonumber, \\
	\label{eq:higher_reg_growth_assumption1}
	\abs{D_u q(u,v)}, \abs{D_v q(u,v)}&\leq C\left(1+\abs{u}^{2/3}+\abs{v}\right) \quad \forall \ u,v \in \R
	\end{align}
        for some $C>0$.
	Then three exists a weak solution $(u,\varphi,v,\mu,\theta)$ to problem \eqref{eq:diffU}--\eqref{eq:theta} such that
	\begin{align*}
	u &\in L^2(0,T;H^{2}(B)), \nonumber \\ v &\in L^\infty(0,T;H^{1}(\Gamma))\cap L^2(0,T;H^{3}(\Gamma)), \nonumber \\  \varphi &\in L^\infty(0,T;H^{2}(\Gamma))\cap L^2(0,T;H^{5}(\Gamma)), \\ \quad \mu &\in L^\infty(0,T;H^{3}(\Gamma))\cap L^2(0,T;H^{5}(\Gamma)), \text{ and } \nonumber \\  \theta &\in L^\infty(0,T;H^{2}(\Gamma))\cap L^2(0,T;H^{3}(\Gamma)) \nonumber
	\end{align*}
	as well as
	\begin{align*}
	&\partial_t u \in L^\infty(0,T;L^2(B))\cap L^2(0,T;H^1(B)),\\
	&\partial_t \varphi \in L^\infty(0,T;H^1(\Gamma))\cap L^2(0,T;H^3(\Gamma)), \text{ and }\\
	&\partial_t \theta \in L^\infty(0,T;L^2(\Gamma))\cap L^2(0,T;H^1(\Gamma)).
	\end{align*}
\end{theorem}
The proof of Theorems \ref{thm:higher_reg} is postponed to Section \ref{sec:proof_reg}.

In the modelling process, the parameter $D$ was the diffusion constant associated with the cytosolic diffusion. This diffusion is often much higher than the lateral diffusion on the cell membrane. The energy bound \eqref{eq:energy_est_lin_growth} implies that $\norm{\nabla u}_{L^2((0,T);L^2(B))} \rightarrow 0$ as $D \rightarrow \infty.$ Hence we expect $u$ to be spatially constant in the limit $D\rightarrow \infty.$  Thus it is reasonable to view the limit $D\rightarrow\infty$ as a reduction of the system \eqref{eq:diffU}--\eqref{eq:theta}. 

 If we formally send $D\rightarrow\infty$ in \eqref{eq:diffU}--\eqref{eq:theta}, we derive the system
\begin{alignat}{2}
\label{eq:flux_red}
\partial_t u &= -\frac{1}{\abs{B}}\int_{\Gamma}q(u,v) &\quad & \text{ for } t \in (0,T],\\
\label{eq:CH1_red}
\pd_t \varphi &= \SL \mu
&&\text{ on } \Gamma \times (0,T],\\
\label{eq:CH2_red}
\mu &= - \eps \SL \varphi + \eps^{-1}W'(\varphi) - \delta^{-1}(2v - 1 - \varphi)
&&\text{ on } \Gamma \times (0,T],\\
\label{eq:v_red}
\pd_t v &= \SL \theta + q(u,v)  \qquad
&&\text{ on } \Gamma \times (0,T],\\
\theta &= \frac2\delta(2v - 1 -
\varphi)&&\text{ on } \Gamma \times (0,T]. \label{eq:theta_red}
\end{alignat}
In the resulting system, $u$ is spatially constant and its evolution in time is governed by an ordinary differential equation which is coupled to the surface diffusion for $v.$ Hence we reduced the coupled bulk-surface system into a system of surface equations with nonlocal contributions, namely via the integral on the right-hand side of \eqref{eq:flux_red}.

Based on the energy estimate \eqref{eq:energy_est_lin_growth}, we have the following rigorous convergence result as $D\rightarrow \infty$.
\begin{proposition}\label{prop:conv_existence_reduced_model}
	Let $\lbrace D_n \rbrace_{n\in\N} \subset (0,\infty)$ be a sequence with $\lim_{n\rightarrow \infty} D_n = \infty$ and denote by $(u^{D_n}, \varphi^{D_n}, v^{D_n}, \theta^{D_n}, \mu^{D_n})$ the weak solution from Theorem \ref{thm:existence} with $D=D_n$ and initial data independent of $n$. Then there exists a subsequence (again denoted by $\lbrace D_n \rbrace_{n\in\N}$) such that
	\begin{align*}
	u^{D_n} &\rightharpoonup u \text{ in } L^2(0,T;H^1(B))\cap H^1\left(0,T;\left(H^{1}(B)\right)'\right) \text{ with } u(t) \in \R \forall t \in (0,T), \\
          \varphi^{D_n} &\rightharpoonup \varphi \text{ in } L^2(0,T;H^1(\Gamma))\cap H^1(0,T;H^{-1}(\Gamma)), \\
          	v^{D_n} &\rightharpoonup v \text{ in } L^2(0,T;H^1(\Gamma))\cap H^1(0,T;H^{-1}(\Gamma)),\\
          \theta^{D_n} &\rightharpoonup \theta \text{ in } L^2(0,T;H^1(\Gamma)), \\
          	\mu^{D_n} &\rightharpoonup \mu \text{ in } L^2(0,T;H^1(\Gamma))
	\end{align*}
	and such that the limit functions are weak solution to the reduced problem \eqref{eq:flux_red}--\eqref{eq:theta_red}, i.e. they fulfil for all $\eta \in L^2(0,T;H^1(\Gamma))$ the equations
	\begin{align*}
	\partial_t u  &= - \frac{1}{\abs{B}} \int_{\Gamma} q(u,v) \qquad \text{for a.e. }t\in (0,T), \\ 
	\int_0^T\dprodH{\G}{\partial_t \varphi}{\eta} &= -\int_0^T\int_{\Gamma}  \SG \mu \cdot \SG \eta, \\
	\int_0^T\int_{\Gamma} \mu \eta &= - \int_0^T\int_{\Gamma} \varepsilon \SG \varphi \cdot \SG \eta + \frac{1}{\varepsilon} \int_0^T\int_{\Gamma} W'(\varphi)\eta - \frac{1}{\delta}\int_0^T\int_{\Gamma} \left( 2v - 1 -\varphi \right)\eta, \\
	\int_0^T \dprodH{\G}{\partial_t v}{\eta} &= -\int_0^T\int_{\Gamma} \SG \theta \cdot \SG \eta +\int_0^T\int_{\Gamma} q(u,v)\eta, \\
	\int_0^T\int_{\Gamma} \theta \eta &= \frac{2}{\delta}\int_{\Gamma} \left( 2v - 1 - \varphi\right)\eta.
	\end{align*}
	The initial values are attained in $L^2(B)$ and $L^2(\G)$ respectively.
\end{proposition}
The proof can be found in Section \ref{sec:proof_D}.

The motivation behind the model \eqref{eq:diffU}--\eqref{eq:theta} was the formation of lipid rafts in biological membranes, i.e. to derive evolution equations that display mesoscale patterns as time evolves. It is thus a natural question to study the qualitative behaviour of the model \eqref{eq:diffU}--\eqref{eq:theta}. In \cite{GKRR}, the authors identified two different qualitative regimes, based on the choice of the exchange term $q.$ 

The inequality \eqref{eq:energy_thermo} allows to identify two different classes of constitutive laws for the exchange term $q.$ Every constitutive law which implies that $\int_{\G} q(\theta-u)$ is non-positive also implies that the energy of the coupled system is decreasing. In this case, we expect the evolution to approach an equilibrium of $\F$ as $t\rightarrow\infty.$ Hence choices for $q$ that lead to a decreasing energy such as $q(u,v)=-c(\theta-u)$ for $c \geq 0$ are referred to as equilibrium cases.

On the other hand, there are choices for $q$ such that $\int_{\G} q(\theta-u)$ does not need to be non-positive. For systems including such an exchange term $q$, it is not reasonable to expect the evolution to converge to an equilibrium point of $\F$ as $t\rightarrow\infty,$ as it is a priori not certain that solutions exist for all times or that $\F$ is bounded in time. Hence these systems are called non-equilibrium models. 

One possible approach leading to a non-equilibrium model is to see the cholesterol attachment to the membrane as a ''reaction'' between free sites on the membrane, namely regions of low membrane-bound cholesterol concentration $v$ and the cytosolic cholesterol, whereas the detachment from the membrane can be considered to be proportional to $v.$ This was suggested in \cite{GKRR} and results in the constitutive choice
\begin{equation}\label{eq:def_q}
q(u,v):= c_1 u(1-v)-c_2v 
\end{equation}
with positive constants $c_1, c_2 \in \R.$  
\begin{remark}\label{rem:boundedness_u_red_model}
		We note that the exchange term $q$ as in \eqref{eq:def_q} does not fulfil the linear growth condition required in Theorem \ref{thm:existence}. However, for $M$ as in Remark \ref{rem:mass_cons} and a smooth, monotone increasing and uniformly bounded function $\eta:\R\rightarrow\R$ such that $\eta(s)=s$ for $\abs{s}\leq M \abs{B}^{-1}$ we can define an alternative exchange term $\tilde{q}$ as
			\[ \tilde{q}(u,v)=c_1u-c_1\eta(u)v-c_2v. \]
			We note that $\tilde{q}$ fulfils the linear growth assumption and coincides with 
			\[ q(u,v)=c_1u(1-v)-c_2v = c_1u - c_1 u v - c_2 v \]
                        if $0\leq u |B| \leq M$. For solutions of the reduced model the latter conditions is preserved in time, which can be seen as follows.
			The mass conservation \eqref{eq:mass_cons} carries over to the reduced model and allows us to find the specific reformulation
			\begin{align}
			\frac{d}{dt} \int_B u(t) \ dx &= -\int_{\G} \tilde{q}(u,v) \nonumber \\
			&= -\frac{c_1\abs{\G}}{\abs{B}}\int_B u(t) \ dx + \left(c_1\eta\left(\frac{1}{\abs{B}}\int_B u(t) \ dx \right)+c_2\right)\left( \int_{\G} v(t) \ d\mathcal{H}^2 \right)\nonumber\\
			&= -\frac{c_1\abs{\G}}{\abs{B}}\int_B u(t) \ dx + \left(c_1\eta\left(\frac{1}{\abs{B}}\int_B u(t) \ dx \right)+c_2\right)\left( M - \int_B u(t) \ dx \right) \label{eq:ode_u}
			\end{align}  
			of the ordinary differential equation for $u.$ Thus the equation is actually independent of $v.$ The right-hand side is strictly positive if $\int_B u(t) \ dx = 0$ and strictly negative if $\int_B u(t) \ dx = M$. Thus we infer that
			\[ u(t) \in [0,\abs{B}^{-1}M] \text{ for all } t\geq 0\] 
			if the initial data was in this range to begin with. Hence for suitable initial data, we actually have 
			\[ \tilde{q}(u,v)=q(u,v) \text{ for all } t \geq 0. \]
			We thus continue to consider the specific form $q(u,v)=c_1u(1-v)-c_2v$ in the reduced model as a prime example for the non-equilibrium case.
\end{remark}

 In \cite{GKRR}, the authors present numerical simulations which allow to compare the qualitative behaviour for the reduced model in the equilibrium and non-equilibrium case. 

In the equilibrium case, the simulations display the saturated lipids clustered in one connected domain, in contrast to the complex patterns observed in the formation of lipid rafts. On the other hand, the non-equilibrium case \eqref{eq:def_q} exhibits the emergence of patterns similar to the formation of lipid rafts, see \cite[Figures 3 and 11]{GKRR}.  As such, the choice \eqref{eq:def_q} will be treated as a prime example of a non-equilibrium system throughout this paper.

Furthermore, it turns out that the reduced system in the non-equilibrium case displays a surprising relationship to the so-called Ohta-Kawasaki system arising in the modeling of diblock copolymers. Depending on the initial value of the lipid concentration $\varphi,$ the almost stationary solutions obtained from the simulation display two distinct classes of patterns, with stripe like patterns emerging if the concentration of saturated and unsaturated lipids is balanced, see \cite[Figure 5]{GKRR}. For less balanced initial values, the experiments show patterns with several circular domains, similar to lipid rafts. 

The stationary states of the Ohta-Kawasaki equations display a similar behaviour. Based on further numerical experiments, Garcke et. al.  conjectured in \cite[Section 3.4]{GKRR} that as $\delta \rightarrow 0,$ solutions to the reduced model in the non-equilibrium case $q=c_1u(1-v)-c_2v$ should approach solutions to the Ohta-Kawasaki equations. 

By the following Theorem \ref{prop:conv_OK}, this is actually true for the mean value free parts of the solutions and a slight modification of the Otha-Kawasaki equation. In the following, $P_\G $ denotes the projection onto the mean value free part, i.e. $P_\G  f := f - \frac{1}{\abs{\Gamma}}\int_{\Gamma} f:=f_\G.$ 

We remark that the existence of weak solutions to the reduced problem is due to Proposition \ref{prop:conv_existence_reduced_model}

\begin{theorem}[Convergence to a Modified Ohta-Kawasaki Equation]\label{prop:conv_OK}~\\
	Let the exchange term $q$ be given as in \eqref{eq:def_q}, i.e. \[ q(u,v)=c_1u(1-v)-c_2v \]
	and let $\lbrace \delta_n \rbrace_{n\in\N} \subset (0,\infty)$ be a sequence with $\lim_{n\rightarrow \infty} \delta_n = 0.$ We denote by $(u^{\delta_n}, \varphi^{\delta_n}, \mu^{\delta_n}, \theta^{\delta_n}, v^{\delta_n})$ a weak solution to the reduced problem \eqref{eq:flux_red}--\eqref{eq:theta_red} from Proposition \ref{prop:conv_existence_reduced_model} with $\delta=\delta_n$. We assume that the initial data is independent of $\delta_n$ and in addition that the initial data for $u$ belongs to $[0,M\abs{B}^{-1}].$ Then there exists a subsequence (again denoted by $\lbrace \delta_n \rbrace_{n\in\N}$) such that $\lbrace u^{\delta_n} \rbrace_{n\in\N}$ and the mean value free functions $(\varphi_\G^{\delta_n}, \mu_\G^{\delta_n}, \theta_\G^{\delta_n})$ fulfil
	\begin{align*}
	u^{\delta_n} &\rightharpoonup u \text{ in } H^1(0,T),\\
	\varphi_\G^{\delta_n} &\rightharpoonup \varphi_\G \text{ in } L^2(0,T;H^1(\Gamma))\cap H^1(0,T;H^{-1}(\Gamma)), \\
	\mu_\G^{\delta_n} &\rightharpoonup \mu_\G \text{ in } L^2(0,T;H^1(\Gamma)), \\
	\theta_\G^{\delta_n} &\rightharpoonup \theta_\G \text{ in } L^2(0,T;H^1(\Gamma)),\\
	\delta_n \partial_t \theta_\G^{\delta_n} &\stackrel{*}{\rightharpoonup} 0 \text{ in } L^2(0,T;H^{-1}(\Gamma))
	\end{align*}
	and such that the limit functions are a weak solution to the modified Ohta-Kawasaki equation
        \begin{align*}
          \partial_t \varphi_\G &= \SL \mu_\G, \\
          \frac{5}{4}\mu_\G &=  -\varepsilon \SL \varphi_\G + \frac{1}{\varepsilon} P_\G W'(\varphi_\G) - \frac{1}{2} \sigma, \\
          \SL \sigma &= \frac{c_1u(t)+c_2}{2} \varphi_\G, \qquad \int_\Gamma \sigma=0, 
        \end{align*}
        where $\sigma := \theta_\G-\frac{1}{2}\mu_\G.$,
        together with 
        \begin{equation*}
        \frac{d}{dt}\int_B u(t) = -\frac{c_1}{\abs{B}}\left( \int_B u(t)  \right)^2 + \left( c_1\frac{M-\abs{\Gamma}}{\abs{B}} - c_2 \right)\int_B u(t) +c_2M \text{ on } (0,T].  
        \end{equation*}
        Here the  modified Ohta-Kawasaki equation is understood in the following weak sense:  For all $\eta \in L^2(0,T;H^1(\Gamma))$ 
	\begin{align*}
	\int_0^T \left\langle\partial_t \varphi_\G, \eta\right\rangle &= -\textbf{}\int_{\Gamma} \SG \mu_\G \cdot \SG \eta, \\
	\frac{5}{4} \int_0^T\int_{\Gamma} \mu_\G \eta &= \int_0^T\int_{\Gamma} \varepsilon \SG \varphi_\G \cdot \SG \eta + \frac{1}{\varepsilon} \int_0^T\int_{\Gamma} W'(\varphi_\G)\eta - \frac{1}{2}\int_0^T\int_{\Gamma} \sigma \eta, \text{ and } \\
	- \int_0^T\int_{\Gamma} \SG \sigma \cdot \SG \eta &= \int_0^T\int_{\Gamma} \frac{c_1 u(t) + c_2}{2} \varphi_\G \eta.
	\end{align*}
      \end{theorem}
      We note that the modification of the Ohta-Kawasaki equation consists in the time dependent coefficient $\frac{c_1u(t)+c_2}{2}$ and the coupling to equation for $u(t)$.
We prove Theorem \ref{prop:conv_OK} in Section \ref{sec:OK}.

Proposition \ref{prop:conv_OK} establishes the connection between the reduced lipid raft model and the Ohta-Kawasaki model on finite time intervals. We conclude the mathematical analysis with results on existence of stationary states and on longtime existence of solutions to  the reduced lipid raft model.

Both results rely on the following growth condition on the exchange term $q.$
\begin{condition}\label{cond:sublin_growth_q}
	Assume that $q: \R \times \R \rightarrow \R$ has sublinear growth, i.e. assume that there exists $\alpha > 1$ such that 
	\begin{equation}\label{eq:sublin_growth}
	\abs{q(u,v)} \leq C\left(1+\abs{u}^{1/\alpha}+\abs{v}^{1/\alpha}\right).
	\end{equation}
\end{condition}
\begin{remark}
	 A similar argument as in Remark \ref{rem:boundedness_u_red_model} shows that we can consider the non-equilibrium case where $q$ is given by \eqref{eq:def_q} even though it does not directly fulfil Condition \ref{cond:sublin_growth_q}. However, one can always modify $q$ with suitable cut-off functions.
\end{remark}
We recall from Remark \ref{rem:mass_cons} that \[ \int_{\G} \varphi \ d\mathcal{H}^2 = m, \quad\quad
\int_B u \ dx + \int_{\G} v \ d\mathcal{H}^2 = M. \] are conserved over time. In order to prove the existence of stationary points to the reduced lipid raft model, we require the following additional condition.  
\begin{condition}\label{cond:mean_value_fix}
	We assume that there exists a continuous operator $S:H^1_{(0)}(\G) \rightarrow \R^2$ such that for all $\tilde{v}\in H^1_{(0)}(\G)=\{u\in H^1(\G): \int_\Gamma u =0\}$ and for any given $M\in\R$ the pair $(\overline{u},\overline{v}):=S(\tilde{v})\in \R^2$ solves
	\begin{align}
	\int_{\Gamma} q(\overline{u},\tilde{v}+\overline{v}) &= 0, \label{eq:cond_mv_1}\\
	\int_B \overline{u} + \int_{\Gamma} \overline{v} &= M.\label{eq:cond_mv_2}
	\end{align} 
	We also write $S_B(\tilde{v})=\overline{u}$ and $S_\G(\tilde{v})=\overline{v}.$
\end{condition}
\begin{remark}
	Condition \ref{cond:mean_value_fix} ensures that in the stationary case the mean values $\frac{1}{\abs{B}}\int_B u$ and $\frac{1}{\abs{\G}}\int_{\G}v$ are determined by the two equations \eqref{eq:cond_mv_1} and \eqref{eq:cond_mv_2}. It should be seen as a condition on $q,$ as the question whether the condition holds actually depends on the form of $q.$ Remark \ref{rem:boundedness_u_red_model} shows that this condition is satisfied for the prime example \eqref{eq:def_q} in the non-equilibrium case \[ q(u,v)=c_1u(1-v)-c_2v, \] since in this case $\int_{\G} q(u,v)$ does not depend on $v$ as can be seen from \eqref{eq:ode_u}.
\end{remark}
\begin{theorem}[Existence of Stationary Solutions]\label{thm:ex_stat_sol}~\\
	Let $m,M \in \R$ be given. Assume that $q: \R \times \R \rightarrow \R$ fulfils Conditions \ref{cond:sublin_growth_q} and \ref{cond:mean_value_fix}.
	Then there exist $u \in \R$ with $0\leq u|B|\leq M$ and functions 
	\[ (\varphi, v, \mu, \theta) \in H^1(\Gamma)\times H^1(\Gamma) \times H^1(\Gamma) \times H^1(\Gamma) \]
	which are weak stationary solutions to the reduced model \eqref{eq:flux_red}--\eqref{eq:theta_red}.
\end{theorem}
The proof is given in Section \ref{sec:proof_stat}.

For the following we assume that $(u,\varphi,v,\theta,\mu)$ is defined for all $t\geq 0$ such that $(u,\varphi,v,\theta,\mu)|_{[0,T]}$ is a weak solution of the reduced model for all  $T\in (0,\infty)$ as in Proposition~\ref{prop:conv_existence_reduced_model}. Existence of such weak solutions for all $t\geq 0$ can be easily proven by replacing the time interval $(0,T)$ for $(u^{D_n},\varphi^{D_n},v^{D_n},\theta^{D_n},\mu^{D_n})$ by $(0,n)$ and using the same arguments as in the proof of Proposition~\ref{prop:conv_existence_reduced_model} together with a suitable diagonal sequence argument.
Boundedness of solutions to the reduced model can be proved provided that the cellular cholesterol concentration $u$ remains uniformly bounded for all times. 
\begin{condition}\label{as:growth_q_longtime}
	We assume that $u \in L^\infty(0,\infty).$ 
\end{condition}
\begin{remark}
For all choices for the exchange term $q$, $u$ is given as the solution to the ordinary differential equation
			\[\frac{d}{dt} \int_B u(t) \ dx = -\int_{\G} q(u,v)\]
			Therefore, Condition \ref{as:growth_q_longtime} is fulfiled if the solution to this equation exists for all times and stays bounded as $t\rightarrow\infty.$
			As we have already discussed in Remark \ref{rem:boundedness_u_red_model}(1) and (2), this is in particular the case for the prime example \eqref{eq:def_q} in the non-equilibrium case with suitable initial values.
\end{remark}
\begin{proposition}\label{prop:longtime}
	Assume that Conditions \ref{cond:sublin_growth_q}  and \ref{as:growth_q_longtime} hold. Then there exists $C>0$ which depends on the initial data but is independent of $t$ such that for almost all $ t\in [0,\infty)$ \[ \F(v(t),\varphi(t)) \leq C. \]
\end{proposition}
This will be proved in Section \ref{sec:long_ex}.

\section{Solutions to the full and reduced model}\label{sec:proofs_ex_reg}
\subsection{Existence of Solutions to the Full Model \eqref{eq:diffU}--\eqref{eq:theta}}\label{sec:proof_ex}
We now prove Theorem \ref{thm:existence} using a typical Galerkin method. Let $\lbrace\omega_i\rbrace_{i \in \mathbb{N}}$ be the family of eigenfunctions of the Laplace-Beltrami-Operator $\SL$ on the surface $\Gamma.$ 
	Analogously, we define $\lbrace\kappa_i\rbrace_{i \in \mathbb{N}}$ to be the family of eigenfunctions to the Laplace-Operator on $B$ with (homogeneous) Neumann boundary conditions. 
	
	We now restrict ourselves to functions of the form
	\begin{alignat*}{2}
	u^N(t,x)&=\sum_{i=1}^N c^i_{u,N}(t)\kappa_i(x),&\qquad
	\varphi^N(t,x)&=\sum_{i=1}^N d^i_{\varphi,N}(t)\omega_i(x),\\
	\mu^N(t,x)&=\sum_{i=1}^N d^i_{\mu,N}(t)\omega_i(x),&
	v^N(t,x)&=\sum_{i=1}^N d^i_{v,N}(t)\omega_i(x),
	\end{alignat*}
	which are elements of the finite dimensional function spaces $V_\Gamma^N:=\spa\left(\lbrace\omega_i\rbrace_{i=1}^N\right)$ and $V_B^N:=\spa\left(\lbrace\kappa_i\rbrace_{i=1}^N\right)$, respectively.
	In accordance with \eqref{eq:theta} we set
	\begin{equation*}
	\theta^N(t,x)=\frac{2}{\delta}\left( 2 d^1_{v,N}(t)-\sqrt{|\Gamma|}-d^1_{\varphi,N}(t) \right)\omega_1 + \frac{2}{\delta}\sum_{i=2}^N \left( 2 d^i_{v,N}(t)-d^i_{\varphi,N}(t)\right)\omega_i.
      \end{equation*}
      Note that $\omega_1= \frac1{\sqrt{|\Gamma|}}$.
	The weak formulation of \eqref{eq:diffU}--\eqref{eq:theta} for test functions $\omega \in V_\Gamma^N$ and $\kappa \in V_B^N$ then reads
	\begin{alignat}{2}
	\label{eq:weak_diffU_approx}
	\int_B \partial_t u^N \kappa &= -D \int_B \nabla u^N \cdot \nabla \kappa - \int_\Gamma q(u^N,v^N)\kappa,\\
	\label{eq:weak_CH1_approx}
	\int_\Gamma \pd_t \varphi^N \omega &= -\int_\Gamma \SG \mu^N \cdot \SG \omega,\\
	\label{eq:weak_CH2_approx}
	\int_\Gamma \mu^N \omega &=  \int_\Gamma \left[ \eps \SG \varphi^N \cdot \SG \omega + \eps^{-1}W'(\varphi^N)\omega - \delta^{-1}(2v^N - 1 - \varphi^N)\omega \right], \\
	\label{eq:weak_v_approx}
	\int_\Gamma \pd_t v^N \omega &= -\int_\Gamma \SG \theta^N \cdot \SG \omega + \int_\Gamma q(u^N,v^N)\omega.
	\end{alignat}
	Choosing $\kappa=\kappa_i$ and $\omega=\omega_i$ in \eqref{eq:weak_diffU_approx}--\eqref{eq:weak_v_approx} above yields a system of ordinary differential equations for the coefficients $c^i_{u,N}, d^i_{\varphi,N}, d^i_{\mu,N}$ and $d^i_{v,N}$, $i=1,\ldots ,n$. 
	
	The system is complemented by initial conditions derived from the initial data $u_0, \varphi_0, v_0.$ To this end, we set the initial conditions for the above system to be $c^i_{u,N}(0) = \int_{B} u_0 \kappa_i,$ $d^i_{\varphi,N}(0) = \int_{\G} \varphi_0 \omega_i$ and so forth. Solutions of this system exist due to the Picard-Lindelöf Theorem on some interval $(0,T_n), T_n > 0$. We simplify the notation and denote these solutions by $c^i_{u,N}, d^i_{\varphi,N}, d^i_{\mu,N}$ and $d^i_{v,N}$. Accordingly, we write \begin{align*}
	u^N(t,x)&=\sum_{i=1}^N c^i_{u,N}(t)\kappa_i(x),\quad
	\varphi^N(t,x)=\sum_{i=1}^N d^i_{\varphi,N}(t)\omega_i(x),\\
	\theta^N(t,x)&=\frac{2}{\delta\sqrt{\abs{\G}}}\left( 2 d^1_{v,N}(t)-\sqrt{|\Gamma|}-d^1_{\varphi,N}(t) \right) + \frac{2}{\delta}\sum_{i=2}^N \left( 2 d^i_{v,N}(t)-d^i_{\varphi,N}(t)\right)\omega_i(x)
	\end{align*}
	for all $t\in (0,T_n)$ and so on. 
	
	We shall now derive estimates that prove that the solutions $c^i_{u,N}, d^i_{\varphi,N}, d^i_{\mu,N}$ and $d^i_{v,N}$ can be extended to an interval $(0,T)$ for every $N \in \N$ and subsequences of $\lbrace u^N \rbrace, \lbrace \varphi^N\rbrace,$ $\lbrace \mu^N\rbrace,$ and $\lbrace v^N\rbrace$ converge to suitable limit functions $u,\mu,\varphi$ and $v$. It remains then to show that the limit functions $u,\mu,\varphi$ and $v$ solve the equations \eqref{eq:diffU}--\eqref{eq:v}.
	
	We begin by noting that $\kappa=u^N$ is an admissible test function in \eqref{eq:weak_diffU_approx}. 
	Choosing $\kappa=u^N$ in \eqref{eq:weak_diffU_approx} yields
	\begin{align*}
	\frac{1}{2}\frac{d}{dt} \left[ \int_B \abs{u^N}^2 \right] = \int_B u^N \partial_t u^N = -D \int_B \abs{\nabla u^N}^2 - \int_{\G} q(u^N,v^N)u^N
	\end{align*}
	where we have used that the time dependent coefficients $c^i_{u,N}(t)$ are solutions to the ODE system above and therefore differentiable in $t$.
	
	Analogously, one has that $\mu^N, \theta^N$ and $-\partial_t \varphi^N$ are elements of $V_\Gamma^N$ and therefore are admissible test functions in \eqref{eq:weak_CH1_approx}--\eqref{eq:weak_v_approx}. 
	Choosing $\omega = -\partial_t \varphi^N$ in \eqref{eq:weak_CH2_approx}, we obtain 
	\begin{align*}
	\int_{\G} -\partial_t \varphi^N \mu^N &= \int_{\G} \left[ -\eps \SG \varphi^N \cdot \SG \left(\partial_t \varphi^N\right) - \eps^{-1}W'(\varphi^N)\partial_t \varphi^N  \right] + \frac{1}{2}\int_{\G} \theta^N\partial_t\varphi^N  \\
	&= -\frac{d}{dt} \left[ \int_{\G}\frac{\eps}{2}\abs{\SG \varphi^N}^2 + \frac{1}{\eps}W(\varphi^N) \right] + \frac{1}{2}\int_{\G} \theta^N\partial_t\varphi^N.
	\end{align*}
	Choosing $\omega=\mu^N$ in \eqref{eq:weak_CH2_approx} leads to
	\begin{align*}
	\int_{\G} \partial_t \varphi^N \mu^N = - \int_{\G} \abs{\SG \mu^N}^2.
	\end{align*}
	Finally, we use that $\partial_t v^N=\frac{\delta}{4}\partial_t \theta+ \frac{1}{2}\partial_t \varphi^N$ to conclude
	\begin{align*}
	\frac{\delta}{2} \frac{d}{dt}\left[\int_{\G} \abs{\theta^N}^2 \right] + \frac{1}{2}\int_{\G} \partial_t \varphi^N \theta^N &= \frac{\delta}{4}\int_{\G} \theta^N \partial_t \theta^N + \frac{1}{2}\int_{\G} \partial_t \varphi^N \theta^N \\&= - \int_{\G} \abs{\SG \theta^N}^2 + \int_{\G} q(u^N,v^N)\theta^N 
	\end{align*}
	from \eqref{eq:weak_v_approx} with $\omega=\theta^N.$ 
	
	We add these four equations to obtain
	\begin{align}\label{eq:energy_step1}
	\frac{d}{dt} \left[ \frac{1}{2} \int_B \abs{u^N}^2 \right. &+ \left. \frac{\eps}{2} \int_\Gamma \abs{\SG \varphi^N}^2 + \frac{1}{\eps}\int_\Gamma W(\varphi^N) + \frac{\delta}{8}\int_\Gamma \abs{\theta^N}^2 \right] \nonumber \\ 
	&+D\int_B \abs{\nabla u^N}^2 + \int_\Gamma \abs{\SG \mu^N}^2 + \int_\Gamma \abs{\SG \theta^N}^2 = \int_\Gamma q(u^N,v^N)(\theta^N-u^N).
	\end{align}
	In order to estimate the right-hand side, we use H\"older's and Young's inequality to estimate
	\begin{align}\label{eq:est_q}
	\abs{\int_\Gamma q(u^N,v^N)(\theta^N-u^N)} 
	\leq& \frac{1}{2}\int_\Gamma \abs{\theta^N-u^N}^2+\frac{1}{2}\int_\Gamma \abs{q(u^N,v^N)}^2 \nonumber \\
	\leq& \int_\Gamma \abs{\theta^N}^2 + \int_\Gamma \abs{u^N}^2 + C \int_{\G} \left( 1 + \abs{u^N}^2 + \abs{v^N}^2 \right) \nonumber \\
	\leq& \int_\Gamma \abs{\theta^N}^2 + C\int_\Gamma \abs{u^N}^2 + C\left(1+\int_\Gamma \abs{v^N}^2\right).
	\end{align}

	Taking into account that $2v^N=\frac{\delta}{2} \theta^N + 1 + \varphi^N$ we derive
	\[ \abs{v^N}^2 \leq C\left(\delta^2 \abs{\theta^N}^2 +  \abs{1+\varphi^N}^2 \right) \] from Young's inequality.
	Since $\abs{1+\varphi^N}^2 \leq C(\eps)\left( 1 + \frac{1}{\eps} W(\varphi^N) \right)$ we thus obtain
	\begin{align*}
	\int_\Gamma \abs{v^N}^2 &\leq C(\delta,\eps) \left( 1 + \frac{\delta}{8}\int_\Gamma \abs{\theta^N}^2+\frac{1}{2\eps}\int_\Gamma W(\varphi^N)\right).
	\end{align*}
	Therefore
	\begin{align}\nonumber
          &\int_\Gamma q(u^N,v^N)(\theta^N-u^N) \\\label{ineq:control_q}
          &\quad \leq C(\delta,\eps) \left[1+\frac{1}{2} \int_\Gamma \abs{u^N}^2 + \frac{\eps}{2} \int_\Gamma \abs{\SG \varphi^N}^2 + \frac{1}{\eps}\int_\Gamma W(\varphi^N) + \frac{\delta}{8}\int_\Gamma \abs{\theta^N}^2 \right].
	\end{align}
	Combining \eqref{eq:energy_step1} and \eqref{ineq:control_q} we arrive at
	\begin{align*}
	\frac{d}{dt} \left[ \frac{1}{2} \int_B \abs{u^N}^2 \right. &+ \left. \frac{\eps}{2} \int_\Gamma \abs{\SG \varphi^N}^2 + \frac{1}{\eps}\int_\Gamma W(\varphi^N) + \frac{\delta}{8}\int_\Gamma \abs{\theta^N}^2 \right] \nonumber \\ 
	&+D\int_B \abs{\nabla u^N}^2 + \int_\Gamma \abs{\SG \mu^N}^2 + \int_\Gamma \abs{\SG \theta^N}^2 \nonumber\\
	&\leq C(\delta,\eps) \left[1+\frac{1}{2} \int_\Gamma \abs{u^N}^2 + \frac{\eps}{2} \int_\Gamma \abs{\SG \varphi^N}^2 + \frac{1}{\eps}\int_\Gamma W(\varphi^N) + \frac{\delta}{8}\int_\Gamma \abs{\theta^N}^2 \right],
	\end{align*}
	which allows us to employ Gronwall's inequality to deduce bounds on $u^N,\phi^N,\mu^N$ and $v^N$ provided we can control $\int_\Gamma \abs{u^N}^2$. 
	
	By \cite[Chapter 2, (2.27)]{LU} the interpolation inequality
	\[ \norm{u}_{L^2(\G)} \leq C \norm{u}_{H^1(B)}^{1/2}\norm{u}_{L^2(B)}^{1/2} \] holds.
	Using this estimate, we immediately find
	\begin{align*}
	\int_\Gamma \abs{u^N}^2 \leq C(a)\norm{u^N}_{L^2(B)}^2 + a\norm{\nabla u^N}_{L^2(B)}^2
	\end{align*}
	for $a>0$ arbitrary small.
	Choosing $a$ small enough, we thus conclude
	\begin{align*}
	\frac{d}{dt} \left[ \frac{1}{2} \int_B \abs{u^N}^2 \right. &+ \left. \frac{\eps}{2} \int_\Gamma \abs{\SG \varphi^N}^2 + \frac{1}{\eps}\int_\Gamma W(\varphi^N) + \frac{\delta}{8}\int_\Gamma \abs{\theta^N}^2 \right] \nonumber \\ 
	&+\frac{D}{2}\int_B \abs{\nabla u^N}^2 + \int_\Gamma \abs{\SG \mu^N}^2 + \int_\Gamma \abs{\SG \theta^N}^2 \nonumber\\
	&\leq C(\delta) \left[1+\frac{1}{2} \int_B \abs{u^N}^2 + \frac{\eps}{2} \int_\Gamma \abs{\SG \varphi^N}^2 + \frac{1}{\eps}\int_\Gamma W(\varphi^N) + \frac{\delta}{8}\int_\Gamma \abs{\theta^N}^2 \right].
	\end{align*}
	We are now in the position to apply Gronwall's inequality and after integrating the above equation in time from $0$ to $T>0$ we deduce
	\begin{align}\label{eq:energy_est}
	\sup_{0\leq t \leq T} \left\lbrace \frac{1}{2} \int_B \abs{u^N}^2 \right. &+ \left. \frac{\eps}{2} \int_\Gamma \abs{\SG \varphi^N}^2 + \frac{1}{\eps}\int_\Gamma W(\varphi^N) + \frac{\delta}{8} \int_\Gamma \abs{\theta^N}^2 \right\rbrace \nonumber \\
	&+\frac{D}{2}\int_0^T \int_B \abs{\nabla u^N}^2 + \int_0^T \int_\Gamma \abs{\SG \mu^N}^2 + \int_0^T \int_\Gamma \abs{\SG \theta^N}^2 \leq C(T).
	\end{align}
	Moreover, choosing $\omega = \omega_1 \equiv \text{const}$ in \eqref{eq:weak_CH2_approx} yields
	\[ \int_{\G} \mu^N = \frac{1}{\eps}\int_{\G} W'(\varphi^N) - \frac{1}{2}\int_{\G} \theta^N.  \]
	Since $\varphi^N$ is bounded in $L^\infty(0,T;H^1(\G))$ by \eqref{eq:energy_est}, the Sobolev embedding theorem in dimension $\dim \G =2$ implies $\varphi^N \in L^\infty(0,T;L^p(\G))$ for all $1\leq p < \infty.$ As $W'(\varphi)=4\varphi^3-\varphi$ and $\abs{\int_{\G} \theta^N} \leq C(\G) \norm{\theta^N}_{L^2(\G)}$ we thus infer that
	\begin{align}\label{eq:control_mean_value_mu}
	\sup_{0\leq t\leq T} \abs{\int_{\G} \mu^N(t) } \leq C\left(\norm{\varphi^N}_{L^\infty(0,T;H^1(\G))} + \norm{\theta^N}_{L^1\infty(0,T;L^2(\G))}\right) \leq C(T)
	\end{align}
	by \eqref{eq:energy_est}. As a result, we obtain $\norm{\mu^N}_{L^2(0,T;H^1(\G))} \leq C(T)$ from Poincar\'{e}'s inequality, \eqref{eq:energy_est}, and \eqref{eq:control_mean_value_mu}.
	
	For any $\tau \in H^1(B)$, there exists $\tau_1 \in \spa{\lbrace \kappa_i \rbrace_{i \in \mathbb{N}}^N}$ such that $\tau_2 := \tau - \tau_1$ is orthogonal to $\spa{\lbrace\kappa_i\rbrace_{i\in\mathbb{N}}^N}$ in $L^2(B)$ as well as in $H^1(B)$. Therefore $\langle \partial_t u^N, \tau\rangle = \int_B \partial_t u^N \tau_1$ and since $\tau_1$ is an admissible test function in \eqref{eq:weak_diffU_approx}, we find
	\begin{align*}
	\abs{\langle \partial_t u^N, \tau\rangle} = \abs{\int_B \partial_t u^N \tau_1} &\leq D \abs{\int_B \nabla u^N \cdot \nabla \tau_1} + \abs{\int_\Gamma q(u^N,v^N)\tau_1} \nonumber \\
	&\leq D \norm{u^N}_{H^1(B)} \norm{\tau_1}_{H^1(B)} + \norm{q(u^N,v^N)}_{L^2(\Gamma)}\norm{\tau_1}_{L^2(\Gamma)}. 
	\end{align*}
	Observe that the continuity of the trace operator ensures $\norm{\tau_1}_{L^2(\Gamma)}\leq C \norm{\tau_1}_{H^1(B)}$ and that $\norm{\tau_1}_{H^1(B)}\leq \norm{\tau}_{H^1(B)}$ since $\lbrace\kappa_i\rbrace_{i\in\mathbb{N}} \subset {H^1(B)}$ is an orthogonal basis. Thus the above inequality implies (after integrating in time)
	\begin{align*}
	\norm{\partial_t u^N}_{L^2(0,T;\left(H^1(B)\right)')} \leq \left( D \norm{u^N}_{L^2(0,T;H^1(B))} + \norm{q(u^N,v^N)}_{L^2(0,T;L^2(\Gamma))}\right)
	\end{align*}
	The norm $\norm{u^N}_{L^2(0,T;H^1(B))}$ can be controlled directly by energy estimate \eqref{eq:energy_est} while similar arguments as in \eqref{eq:est_q} allow us to deduce that $\norm{q(u^N,v^N)}_{L^2(0,T;L^2(\Gamma))}$ is bounded by the constant $C(T)$ from $\eqref{eq:energy_est}.$ 
	The embedding $H^1(B) \hookrightarrow H^s(B)$ is compact for all $1/2 < s < 1$. The Aubin-Lions theorem \cite[Corollary 2]{JS} applied to $H^1(B) \hookrightarrow H^s(B) \hookrightarrow H^{-1}(B)$	allows us to deduce the relative compactness of $\lbrace u_k \rbrace$ in $L^2([0, T]\times H^s(B)), $ i.e., after possibly extracting a subsequence, the strong convergence $u_k \rightarrow u$ in $L^2(0,T;H^s(B)).$ By the continuity of the trace operator, we deduce $\tr(u_k) \rightarrow \tr(u)$ in $L^2([0, T] \times \Gamma)$.
	
	Analogously to the bound on $\norm{\partial_t u^N}_{L^2(0,T;\left(H^1(B)\right)')}$ , we obtain $\norm{\partial_t \varphi^N}_{L^2(0,T;H^{-1}(\Gamma))} \leq C(T)$ and $\norm{\partial_t v^N}_{L^2(0,T;H^{-1}(\Gamma))} \leq C(T)$.
	
	The Aubin-Lions theorem \cite[Corollary 2]{JS} applied for the Gelfand triple $H^1(\Gamma) \hookrightarrow L^2(\Gamma) \hookrightarrow H^{-1}(\Gamma)$ allows us to deduce the relative compactness of $\lbrace v^N \rbrace$ and $\lbrace \varphi^N \rbrace$ in $L^q(0,T;H^s(\Gamma))$ for every $1\leq q <\infty$, $s\in [0,1)$ and consequently (up to the extraction of a subsequence) the convergence of $v^N$ and $\varphi^N$ pointwise almost everywhere in $\Gamma \times [0,T].$ 

	Summing up our results, we thus deduce that there exist subsequences (which we also denote by $(u^N,\varphi^N,\mu^N,\theta^N,v^N)$ such that
	\begin{align}
	u^N \rightharpoonup u &\text{ in } L^2(0,T;H^1(B)), \\ u^N \rightarrow u &\text{ in } L^2(0,T;H^s(B)), 0 < s < 1,\label{conv:u1}\\
	\tr(u^N) \rightarrow \tr(u) \text{ in } L^2(0,T;L^2(\Gamma)) &\text{ and } \tr(u^N) \rightarrow \tr(u) \text{ a.e. in } \Gamma_T,\\
	\varphi^N \rightharpoonup^\ast \varphi &\text{ in } L^\infty(0,T;H^1(\Gamma)),\\  \varphi^N \rightarrow \varphi &\text{ in } L^q(0,T;H^s(\Gamma)), 1\leq q<\infty, 0\leq s<1, \label{conv:phi}\\
	\mu^n \rightharpoonup \mu &\text{ in } L^2(0,T;H^1(\Gamma)), \\
	\theta^n \rightharpoonup \theta &\text{ in } L^2(0,T;H^1(\Gamma)), \\
	v^N \rightharpoonup v \text{ in } L^2(0,T;H^1(\Gamma)) &\text{ and } v^N \rightarrow v \text{ in } L^2(0,T;L^2(\Gamma)),\\
	v^N \rightarrow v \text{ a.e. in } \Gamma_T, \label{conv:v2}
	\end{align}
	as $N\rightarrow\infty$ while the time derivatives fulfil
	\begin{align}
	\partial_t u^N \rightharpoonup \partial_t u \text{ in } L^2(0,T;\left(H^1(B)\right)'), \label{conv:dt_u}\\
	\partial_t \varphi^N \rightharpoonup \partial_t \varphi \text{ in } L^2(0,T;H^{-1}(\Gamma)), \\
	\partial_t v^N \rightharpoonup \partial_t v \text{ in } L^2(0,T;H^{-1}(\Gamma)).\label{conv:dt_v}
	\end{align}
	Using \eqref{eq:growth_cond_q}, $|W'(s)|\leq C(|s|^3+1)$, and the theory of Nemytskii operators (cf.\ e.g.\ \cite[Theorem 10.58]{ReRo}) we obtain the convergence
	\begin{align}
          q(u^N,v^N) &\rightarrow_{N\to\infty} q(u,v) \text{ in }  L^2(0,T;L^2(\Gamma)),\\\label{conv:W}
          W'(\varphi^N) &\rightarrow_{N\to\infty} W'(\varphi) \text{ in }  L^2(0,T;L^2(\Gamma)).
        \end{align}

	Let $N_0 \in \mathbb{N}$ be arbitrary and consider the weak formulation of equations \eqref{eq:diffU}--\eqref{eq:theta} for test functions $\omega \in C^1_c([0,T];V^{N_0}_\Gamma)$ and $\kappa \in C^1_c([0,T];V^{N_0}_B)$. From the convergence results in \eqref{conv:u1}--\eqref{conv:v2} and \eqref{conv:dt_u}--\eqref{conv:W} we derive that the limit functions $u,v,\varphi,\mu$ and $\theta$ fulfil 
	\begin{alignat*}{2}
	\int_0^T \langle\partial_t u,\kappa\rangle_{\left(H^1(B)\right)',H^1(B)} &= -D \int_0^T \int_B \nabla u \cdot \nabla \kappa - \int_0^T \int_\Gamma q(u,v)\kappa,\\
	\int_0^T \dprodH{\G}{\pd_t \varphi}{\omega} &= -\int_0^T \int_\Gamma \SG \mu \cdot \SG \omega,\\
	\int_0^T \int_\Gamma \mu \omega &=  \int_0^T \int_\Gamma \left[ \eps \SG \varphi \cdot \SG \omega + \eps^{-1}W'(\varphi)\omega - \delta^{-1}(2v - 1 - \varphi)\omega \right], \\
	\int_0^T \dprodH{\G}{\pd_t v}{\omega} &= -\int_0^T \int_\Gamma \SG \theta \cdot \SG \omega + \int_0^T \int_\Gamma q(u,v)\omega,
	\end{alignat*}
	first for all $\omega \in C^1_c([0,T];V^{N_0}_\Gamma)$ and $\kappa \in C^1_c([0,T];V^{N_0}_B)$, but since $N_0 \in \N$ was arbitrary also for all $\omega \in C^1_c([0,T];\bigcup_{N\in\N}V^{N}_\Gamma)$ and $\kappa \in C^1_c([0,T];\bigcup_{N\in\N}V^{N}_B).$ Using that $\bigcup_{N\in\N}V^{N}_\Gamma$ and $\bigcup_{N\in\N}V^{N}_B$ are dense in $H^1(\Gamma)$ and $H^1(B)$ respectively, we deduce that these equations actually hold for all test functions $\omega \in H^1(0,T;H^1(\Gamma))$ and $\kappa \in H^1(0,T;H^1(B))$.
	
	Now observe that for all $\kappa \in C^1([0,T];V^{N_0}_B)$ such that $\kappa(T)=0$
	\begin{align*}
	&\int_B u(x,0)\kappa(x,0) \ dx \\ =& -\int_0^T \dprodH{B}{\partial_t u(\cdot,t)}{\kappa(\cdot,t)} \ dt - \int_0^T \dprodH{B}{u(\cdot,t)}{\pd_t \kappa(\cdot,t)} \ dt \\
	=& D \int_0^T \int_B \nabla u \cdot \nabla \kappa + \int_0^T \int_\Gamma q(u,v)\kappa - \int_0^T \dprodH{B}{u(\cdot,t)}{\pd_t \kappa(\cdot,t)} \ dt \\
	=& \lim_{N\rightarrow\infty}\left( D \int_0^T \int_B \nabla u^N \cdot \nabla \kappa + \int_0^T \int_\Gamma q(u^N,v^N)\kappa - \int_0^T \dprodH{B}{u^N(\cdot,t)}{\pd_t \kappa(\cdot,t)} \ dt \right).
	\end{align*}
	By \eqref{eq:weak_diffU_approx} we deduce that
	\begin{align*}
	&\int_0^T \dprodH{B}{u^N(\cdot,t)}{\pd_t \kappa(\cdot,t)} \ dt \\ =& -\int_0^T \dprodH{B}{\pd_t u^N(\cdot,t)}{\kappa(\cdot,t)} \ dt  - \dprodH{\G}{u^N(\cdot,0)}{\kappa(\cdot,0)} \\
	=& D \int_0^T \int_B \nabla u^N \cdot \nabla \kappa \ dt + \int_0^T\int_\Gamma q(u^N,v^N)\kappa \ dt - \dprodH{\G}{u^N(\cdot,0)}{\kappa(\cdot,0)}.
	\end{align*}
	Hence 
	\begin{align*}
	&\int_B u(x,0)\kappa(x,0) \ dx \\ &= \lim_{N\rightarrow\infty}\left( D \int_0^T \int_B \nabla u^N \cdot \nabla \kappa + \int_0^T \int_\Gamma q(u^N,v^N)\kappa - \int_0^T \dprodH{B}{u^N(\cdot,t)}{\pd_t \kappa(\cdot,t)} \ dt \right) \\ 
	&=  \lim_{N\rightarrow\infty} \dprodH{\G}{u^N(\cdot,0)}{\kappa(\cdot,0)} = \int_B u_0(x)\kappa(x,0) \ dx
	\end{align*}
	for $\kappa \in C^1([0,T];V^{N_0}_B)$ with $\kappa(T)=0$ and $N_0 \in \N$ arbitrary. Thus $u(\cdot,0) = u_0(\cdot)$ in $L^2(B).$ In the same way, we deduce $\varphi(\cdot,0)=\varphi_0(\cdot)$ and $v(\cdot,0)=v_0(\cdot)$ in $L^2(\G).$
	
	Finally, \eqref{eq:energy_est} is uniform in $N$ and therefore implies \eqref{eq:energy_est_lin_growth}.
\subsection{Existence of Solutions to the Reduced Model and the Limit Process $D\rightarrow\infty$}\label{sec:proof_D}

After we proved the necessary estimate \eqref{eq:energy_est_lin_growth} rigorously in Theorem \ref{thm:existence}, we are now in the position to prove Proposition \ref{prop:conv_existence_reduced_model}, thus establishing the connection between the full model \eqref{eq:diffU}--\eqref{eq:theta} and the reduced model \eqref{eq:flux_red}--\eqref{eq:theta_red} rigorously. Note that Proposition \ref{prop:conv_existence_reduced_model} not only assures the convergence of solutions to the full model as $D\rightarrow\infty$ but also gives an existence result for solutions to the reduced model. 

\begin{proof}[Proof of Proposition \ref{prop:conv_existence_reduced_model}]
	According to \eqref{eq:energy_est_lin_growth} the solutions $(u^{D_n}, \varphi^{D_n}, \mu^{D_n}, \theta^{D_n}, v^{D_n})$ fulfil 
	\begin{align}\label{eq:energy_est_reducing_proof}
	\sup_{0\leq t \leq T} \left\lbrace \frac{1}{2} \int_B \abs{u^{D_n}}^2 \right. &+ \left. \frac{\eps}{2} \int_\Gamma \abs{\SG \varphi^{D_n}}^2 + \frac{1}{\eps}\int_\Gamma W(\varphi^{D_n}) + \frac{\delta}{8}\int_\Gamma \abs{\theta^{D_n}}^2 \right\rbrace \nonumber \\
	&+\frac{D_n}{2}\int_0^T\int_B \abs{\nabla u^{D_n}}^2 + \int_0^T\int_\Gamma \abs{\SG \mu^{D_n}}^2 + \int_0^T\int_\Gamma \abs{\SG \theta^{D_n}}^2 \leq C(T).
	\end{align}
	We exploit \eqref{eq:energy_est_reducing_proof} to deduce
	\begin{equation}\label{eq:u_spat_const_limit}
	\int_0^T\int_B \abs{\nabla u^{D_n}}^2 \leq \frac{C(T)}{D_n} \rightarrow 0 \text{ as } n\rightarrow \infty.
	\end{equation} 
	Moreover, choosing a spatially constant $\tau=\tau(t)$ in \eqref{eq:weak_form_diffU} yields
	\[ \abs{\frac{d}{dt}\int_B u^{D_n}} \leq \abs{\int_\Gamma q(u^{D_n},v^{D_n})} \leq \norm{q(u^{D_n},v^{D_n})}_{L^2(\Gamma)}.   \]
	Together with $\sup_{0\leq t \leq T} \abs{\int_B u^{D_n}} \leq C(T)$ from \eqref{eq:energy_est_reducing_proof} we deduce that $\int_B u^{D_n} \ dx $ is bounded in $H^1(0,T).$ Thus Poincar\'{e}'s inequality implies the convergence
	\[u^{D_n} \rightarrow u \text{ in } L^2(0,T;H^1(B))\]
	and by \eqref{eq:u_spat_const_limit} we have $\nabla u \equiv 0.$ Thus the limit function $u$ is constant in the space variables.

Furthermore, \eqref{eq:CH1} and \eqref{eq:v} imply 
	\[ \norm{\partial_t \varphi^{D_n}}_{L^2(0,T;H^{-1}(\G))} \leq C(T) \text{ and } \norm{\partial_t v^{D_n}}_{L^2(0,t;H^{-1}(\G))} \leq C(T)  \]
	by similar arguments as in the proof of Theorem \ref{thm:existence}. Thus the time derivatives fulfil
	\begin{align*}
	\partial_t \varphi^{D_n} \rightharpoonup \partial_t \varphi \text{ in } L^2(0,T;H^{-1}(\Gamma)), \\
	\partial_t v^{D_n} \rightharpoonup \partial_t v \text{ in } L^2(0,T;H^{-1}(\Gamma)).
	\end{align*}
	The estimate \eqref{eq:energy_est_reducing_proof} yields additionally the existence of subsequences (again denoted by $D_n$) such that
	\begin{align*}
	u^{D_n} \rightharpoonup u &\text{ in } L^2(0,T;H^1(B)), \\
	\tr(u^{D_n}) \rightarrow \tr(u) \text{ in } L^2(0,T;L^2(\Gamma)) &\text{ and } \tr(u^{D_n}) \rightarrow \tr(u) \text{ a.e. in } \Gamma_T,\\
	\varphi^{D_n} \rightharpoonup^\ast \varphi &\text{ in } L^\infty(0,T;H^1(\Gamma))\\ \varphi^{D_n} \rightarrow \varphi &\text{ in } L^q(0,T;H^s(\Gamma)), \forall 1\leq q<\infty,0\leq s<1, \\
	\mu^{D_n} \rightharpoonup \mu &\text{ in } L^2(0,T;H^1(\Gamma)), \\
	\theta^{D_n} \rightharpoonup \theta &\text{ in } L^2(0,T;H^1(\Gamma)), \\
	v^{D_n} \rightharpoonup v \text{ in } L^2(0,T;H^1(\Gamma)) &\text{ and } v^{D_n} \rightarrow v \text{ in } L^2(0,T;L^2(\Gamma)),\\
	v^{D_n} \rightarrow v &\text{ a.e. in } \Gamma_T.
	\end{align*}
	
	The strong convergences $v^{D_n} \rightarrow v$ and $\varphi^{D_n} \rightarrow \varphi$ in $L^2(0,T;L^2(\G))$ here are a consequence of the Aubin-Lions theorem. We remark that these arguments are completely analogous to the proof of Theorem \ref{thm:existence} and we thus omit some details.
		
	It remains to discuss the limit process within the equations. Again, we refer to the proof of Theorem \ref{thm:existence} for the details since the arguments in both cases are completely analogous. As before, we use the theory of Nemytskii operators (see \cite[Theorem 10.58]{ReRo}) to derive
	\begin{equation*}
	q(u^{D_n},v^{D_n}) \rightarrow_{N\to\infty} q(u,v)\text{ and }W'(\varphi^{D_n})\rightarrow_{N\to\infty} W'(\varphi) \text{ in }  L^2(0,T;L^2(\Gamma)).
      \end{equation*}
      	Hence we can take the limit in \eqref{eq:weak_form_time_der_phi}--\eqref{eq:weak_form_theta}. 
	We choose a spatially constant test function in \eqref{eq:weak_form_diffU} and use this information to take the limit $n\rightarrow\infty$ to derive \eqref{eq:flux_red}. 
\end{proof}
	
\subsection{Higher Regularity} \label{sec:proof_reg}
We conclude this section with the proof of Theorem \ref{thm:higher_reg}. Before we prove the theorem, we state the following consequence from the growth assumptions \eqref{eq:higher_reg_growth_assumption1} on $D_u q$ and $D_v q.$ 
\begin{lemma}\label{l:higher_reg_growth_to_int}
	Let $u: B \rightarrow \R$ and $v: \G \rightarrow \R$ be bounded in $L^2(0,T;H^1(B))\cap L^\infty(0,T;L^2(B))$  and in $L^2(0,T;H^1(\G))\cap L^\infty(0,T;L^2(\G))$ respectively and assume that $q$ fulfils condition \eqref{eq:higher_reg_growth_assumption1}. Then 
	\begin{equation}\label{eq:higher_reg_assumption1}
	D_u q(u,v), D_vq(u,v) \in L^6(0,T;L^3(\Gamma)) \cap L^4(0,T;L^4(\Gamma)).
	\end{equation}
\end{lemma}
\begin{proof}
	We only prove the assertion of the lemma for $D_u q(u,v)$ since both $D_u q(u,v)$ and $D_v q(u,v)$ fulfil the same growth property.

	We start with the observation that for $s \in (0,1)$ the space $H^s(B)$ is an interpolation space between $L^2(B)$ and $H^1(B)$ of exponent $s$ and accordingly fulfils \[\norm{f}_{H^s(B)} \leq C \norm{f}_{L^2(B)}^{1-s} \norm{f}^s_{H^1(B)}\] for all $f \in H^1(B),$ see \cite[Section 7.4.5]{TR_FS2}. Together with Hölder's inequality we thus infer for $u \in L^2(0,T;H^1(B))\cap L^\infty(0,T;L^2(B))$ and $p \geq 2$ that $u \in L^p(0,T;H^{2/p}(B)).$ 
	
	For $2 \leq p < 4$ the Trace Theorem \cite[Theorem 7.39]{AD} hence allows us to deduce $u \in L^p(0,T;H^{2/p-1/2}(\G)).$ 
	
	Similarly, $v \in L^2(0,T;H^1(B))\cap L^\infty(0,T;L^2(B))$ implies that $v \in L^p(0,T;H^{2/p}(\G))$ for all $p \geq 2$ and in particular $v \in L^4(0,T;L^4(\G))$ for $p = 4$ since $H^{1/2}(\G)\hookrightarrow L^4(\G).$ 
	
	We use this considerations to estimate
	\begin{align*}
	\int_0^T \left( \int_{\G} \abs{D_uq(u,v)}^4 \right) \leq& C \int_0^T \left( \int_{\G} \abs{1+\abs{u}^{2/3}+\abs{v}}^4 \right) 
	\leq C \int_0^T \left( \int_{\G} 1 + \abs{u}^{8/3} + \abs{v}^4 \right) \\
	&\leq C(\G,T) + C\left( \int_0^T \int_{\G} \abs{u}^{8/3} \right) + C\left( \int_0^T \int_{\G} \abs{v}^4  \right),
	\end{align*}
	where the last term is finite by the considerations on $v$ above. As before, we find $u \in L^p(0,T;H^{2/p-1/2}(\G))$ for $2 \leq p < 4.$ By the Sobolev embedding theorem, we thus have $u \in L^p(0,T;L^{\frac{4p}{3p-4}}(\G))$ which for $p=\frac{8}{3}$ gives $u \in L^{8/3}(0,T;L^{8/3}(\G)).$ Hence the second term is finite as well, implying $D_u q(u,v) \in L^4(0,T;L^4(\G)).$ 
	
	Analogously, we find
	\begin{align*}
	\int_0^T \left( \int_{\G} \abs{D_uq(u,v)}^3 \right)^{6/3} \leq& C \int_0^T \left( \int_{\G} \abs{1+\abs{u}^{2/3}+\abs{v}}^3 \right)^2 
	\leq C \int_0^T \left( \int_{\G} 1 + \abs{u}^2 + \abs{v}^3 \right)^2 \\
	&\leq C(\G,T) + C\left( \int_0^T \left( \int_{\G} \abs{u}^2 \right)^2 \right) + C\left( \int_0^T \left( \int_{\G} \abs{v}^3 \right)^2 \right)
	\end{align*}
	Using again the interpolation estimate $\norm{u}_{L^2(\G)} \leq C \norm{u}_{H^1(B)}^{1/2}\norm{u}_{L^2(B)}^{1/2}$ (see \cite[Chapter 2, (2.27)]{LU}) and integrating $\norm{u(t)}^4_{L^2(\G)}$ in time thus yields 
	\[ \norm{u}_{L^4(0,T;L^2(\G))} \leq C \norm{u}_{L^\infty(0,T;L^2(B))} \norm{u}_{L^2(0,T;H^1(B))},\]
	which is bounded for $u \in L^2(0,T;H^1(B)\cap L^\infty(0,T;L^2(B))$. Therefore the second term on the right-hand side in the foregoing estimate is finite. As above, $v \in L^p(0,T;H^{2/p}(\G))$ for all $p \geq 2$ and in particular for $p=6.$ By the Sobolev embedding theorem we have $H^{1/3}(\G) \hookrightarrow L^3(\G)$ and thus the third term above is finite. Altogether, we obtain $D_u q(u,v) \in L^6(0,T;L^3(\Gamma)).$
\end{proof}
\begin{proof}[Proof of Theorem \ref{thm:higher_reg}]
	The proof of Theorem \ref{thm:higher_reg} can be divided into three steps. In the first step, we consider the approximate solutions $(u^N,v^N,\varphi^N,\mu^N,\theta^N)$ from the proof of the existence theorem (Theorem \ref{thm:existence}) and prove regularity estimates for these functions and their time derivatives. Secondly, we show that the limit functions of these time derivatives as $N \rightarrow \infty$ converge to solutions of the linearised model. This step is summarized in Lemma \ref{l:higher_reg_sol_lin_sys}. Finally, we derive higher regularity for the full system from the additional information gathered from the linearised system.
	
	\textit{First Step: Higher regularity for the approximate solutions.}
	We recall the proof of Theorem \ref{thm:existence} and let $(u^N,v^N,\varphi^N,\mu^N,\theta^N)$ denote the subsequence of solutions to the approximate problem \eqref{eq:weak_diffU_approx}--\eqref{eq:weak_v_approx} which converges to $(u,\varphi,v,\mu,\theta)$. Let $P_N^\G$ denote the orthogonal projection in $H^1(\Gamma)$ onto $V^N_\Gamma,$ where $V^N_\Gamma$ is defined as in the proof of Theorem \ref{thm:existence}. We remark that $P_N^\G$ is also orthogonal with respect to the inner product on $L^2(\G).$
	 
	Thus $\varphi^N, \mu^N$ and $\theta^N \in V_\Gamma^N$ fulfil
	\begin{equation*}
	\int_{\Gamma} \mu^N \omega = \varepsilon \int_{\Gamma} \SG \varphi^N \cdot \SG \omega + \frac{1}{\varepsilon}\int_{\Gamma} P_N^\G W'(\varphi^N) \omega - \int_{\Gamma} \frac{\theta^N}{2} \omega
	\end{equation*}
	for all $\omega \in V^N_\Gamma.$ By the orthogonal decomposition $H^1(\Gamma) = V_\Gamma^N \oplus \left(V_\Gamma^N\right)^\perp$ this equation also holds for all test functions $\omega \in H^1(\Gamma),$ which implies that $\varphi^N$ is a weak solution to the elliptic equation
	\begin{equation}\label{eq:elliptic_varphi_N}
	- \varepsilon \SL\varphi^N = \mu^N + \frac{\theta^N}{2} - \frac{1}{\varepsilon}P_N^\G W'(\varphi^N).
	\end{equation} 
	Furthermore, the energy estimate \eqref{eq:energy_est} together with \eqref{eq:control_mean_value_mu} yields \[ \mu^N,\theta^N \in L^2(0,T;H^1(\Gamma)) \text{ and } \varphi^N \in L^\infty(0,T;H^1(\Gamma)). \]
	In particular,
	\begin{equation}\label{eq:int_phi_N}
	\varphi^N \in L^\infty(0,T;L^p(\Gamma)) \text{ is bounded for all } 1\leq p < \infty 
	\end{equation} 
	by the Sobolev embedding theorem in dimension $\dim \Gamma = 2.$
	
	Observe that therefore every polynomial in $\varphi^N$ is an element of $ L^\infty(0,T;L^p(\Gamma))$ for all $1\leq p < \infty.$ We will exploit this property in particular with respect to $W'(\varphi^N), W''(\varphi^N),$ and $W'''(\varphi^N)$ since these terms grow at most polynomial in $\varphi^N.$ For example, $W'$ fulfils $\abs{W'(s)} \leq C (\abs{s}^3+1)$ for some $C>0$ and $s\in\R$.    
	
	As a first application, we directly deduce the boundedness of $W'(\varphi^N)$ in $L^2(0,T;L^2(\Gamma)).$ Hence the right-hand side in \eqref{eq:elliptic_varphi_N} is in $L^2(0,T;L^2(\Gamma))$. Elliptic theory, see e.g. \cite[Theorem 8.8, Theorem 8.12]{GT}, thus implies that the solution $\varphi^N$ to \eqref{eq:elliptic_varphi_N} fulfils $\varphi^N$ is bounded in $L^2(0,T;H^2(\G)).$ 	We remark that all these estimates are derived from the energy estimate \eqref{eq:energy_est}, which is uniform in $N.$ Hence we conclude that $\lbrace \varphi^N \rbrace_{N\in\N} \subset L^2(0,T;H^2(\G))$ is uniformly bounded in $N.$ 
	
	Additionally, the Sobolev embedding and $\varphi^N \in L^2(0,T;H^2(\G))$ directly yield
	\begin{equation}\label{eq:int_SG_phi_N}
	\norm{\varphi^N}_{L^2(0,T;W^{1,p}(\G))} \leq C \text{ for all } 1\leq p < \infty.
	\end{equation} 
	We calculate
	\begin{align*}
	\int_0^T \int_{\Gamma} \abs{\SG\left(W'(\varphi^N)\right)}^2 &\leq  \int_0^T \int_{\Gamma} \abs{W''(\varphi^N)}^2\abs{ \SG\varphi^N}^2 \\
	&\leq  \int_0^T \left( \int_{\G} \left(\abs{W''(\varphi^N)}\right)^4 \right)^{1/2} \left( \int_{\G} \abs{ \SG\varphi^N}^4 \right)^{1/2} \\
	&\leq C \left(\norm{W''(\varphi^N)}_{L^\infty(0,T;L^4(\G))}+1\right) \norm{\SG \varphi^N}_{L^2(0,T;L^4(\G))}^2, 
	\end{align*}
	which yields a uniform bound in $N$ for $\norm{W'(\varphi^N)}_{L^2(0,T;H^1(\Gamma))}$ by \eqref{eq:energy_est} and the foregoing discussion.
	Moreover, $\norm{P_N^\G}_{\mathcal{L}(H^1(\Gamma))} \leq 1$ implies \[\norm{P_N^\G W'(\varphi^N)}_{L^2(0,T;H^1(\Gamma))} \leq \norm{W'(\varphi^N)}_{L^2(0,T;H^1(\Gamma))}, \]
	showing that the right-hand side in $\eqref{eq:elliptic_varphi_N}$ belongs to $L^2(0,T;H^1(\Gamma))$ and that the corresponding bound is uniform in $N.$ As a direct consequence, we infer
	\begin{equation*}
	\norm{\varphi^N}_{L^2(0,T;H^3(\G)) \cap L^\infty(0,T;H^1(\Gamma)))} \leq C
	\end{equation*} uniformly in $N$ by using standard elliptic theory, see e.g. \cite[Theorem 8.8, Theorem 8.12]{GT}.
	
	We remark for later use that the same argument applied to equation \eqref{eq:CH2} also implies
	\begin{equation}\label{eq:int_SG_phi}
	\varphi \in L^2(0,T;H^3(\G)) \cap L^\infty(0,T;H^1(\Gamma)) \hookrightarrow L^2(0,T;W^{2,p}(\G)) \text{ for all } 1\leq p < \infty.
	\end{equation}
	
	Next we differentiate the equations \eqref{eq:weak_diffU_approx}--\eqref{eq:weak_v_approx} in time. Note that the approximate solutions $u^N, \varphi^N, v^N, \mu^N,$ and $\theta^N$ were all constructed from solutions to a system of ordinary differential solutions, i.e. they are all differentiable in $t.$ We introduce the notation
	\[ \ut = \partial_t u^N, \phit = \partial_t \varphi^N, \vt=\partial_t v^N, \mut=\partial_t \mu^N \text{ and } \thetat = \partial_t \theta^N. \]
	The tuple $(\ut,\phit,\vt,\mut,\thetat)$ solves for all $\kappa \in V_B^N$ and all $\omega \in V_\G^N$
	\begin{alignat}{2}
	\label{eq:weak_dt_diffU_approx}
	\int_B \partial_t \ut \kappa &= -D \int_B \nabla \ut \cdot \nabla \kappa - \int_\Gamma \frac{d}{dt}\left(q(u^N,v^N)\right)\kappa,\\
	\label{eq:weak_dt_CH1_approx}
	\int_\Gamma \pd_t \phit \omega &= -\int_\Gamma \SG \mut \cdot \SG \omega\\
	\label{eq:weak_dt_CH2_approx}
	\int_\Gamma \mut \omega &=  \int_\Gamma \left[ \eps \SG \phit \cdot \SG \omega + \eps^{-1}W''(\varphi^N) \phit \omega - \frac{\thetat}{2}\omega \right] \\
	\label{eq:weak_dt_theta_approx}
	\int_{\Gamma} \thetat \omega &= \frac{2}{\delta}\int_{\Gamma} (2\vt - \phit)\omega \\
	\label{eq:weak_dt_v_approx}
	\int_\Gamma \frac{\delta}{4}\pd_t \thetat \omega + \int_\Gamma \frac{1}{2}\pd_t \phit \omega &= -\int_\Gamma \SG \thetat \cdot \SG \omega + \int_\Gamma \frac{d}{dt}\left(q(u^N,v^N)\right)\omega. 
	\end{alignat}
	\begin{lemma}\label{l:higher_reg_energy_time_der}
		Let $(\ut,\phit,\vt,\mut,\thetat)$ be defined as above. Under the assumptions of Theorem \ref{thm:higher_reg} the estimate
		\begin{align}
		\sup_{t\in(0,T)} &\left\lbrace \frac{\varepsilon}{2}\norm{\SG \phit}_{L^2(\Gamma)}^2 + \frac{\delta}{8}\norm{\thetat}_{L^2(\Gamma)}^2 + \frac{1}{2} \norm{\ut}_{L^2(B)}^2 \right\rbrace \nonumber \\ &+ \int_{\Gamma} \abs{\SG \mut}^2 +\int_{\Gamma}\abs{\SG \thetat}^2 + D \int_B \abs{\nabla \ut}^2 \leq C(T). \label{eq:energy_est_dt}
		\end{align}
		holds. The estimate is uniform in $N.$
	\end{lemma}
	\begin{proof}[Proof of Lemma \ref{l:higher_reg_energy_time_der}]
		We choose $\omega=\ut$ as a test function in \eqref{eq:weak_dt_diffU_approx}, $\omega=\mut$ in \eqref{eq:weak_dt_CH1_approx}, $\omega=\partial_t \phit$ in \eqref{eq:weak_dt_CH2_approx} and $\omega = \thetat$ in \eqref{eq:weak_dt_v_approx}. We add these equations to deduce
		\begin{align}\label{eq:higher_reg_energy_eq}
		\frac{\varepsilon}{2} \frac{d}{dt}\int_{\Gamma} \abs{\SG \phit}^2 + &\frac{\delta}{8}\frac{d}{dt}\int_{\Gamma} \abs{\thetat}^2 + \int_{\Gamma} \abs{\SG \mut}^2 +\int_{\Gamma}\abs{\SG \thetat}^2 \nonumber + \frac{1}{2}\frac{d}{dt}\int_{B} \abs{\ut}^2 + D \int_B \abs{\nabla \ut}^2 \nonumber \\ & = -\frac{1}{\varepsilon}\int_{\Gamma} W''(\varphi^N)\phit\partial_t \phit + \int_{\Gamma} \frac{d}{dt}\left(q(u^N,v^N)\right)\left(\thetat-\ut\right). 
		\end{align}
		To estimate the right-hand side in \eqref{eq:higher_reg_energy_eq} we first compute for any $\gamma > 0$
		\begin{align}\label{eq:higher_reg_W_est_step1}
		\abs{\frac{1}{\varepsilon}\int_\Gamma W''(\varphi^N)\phit \partial_t \phit} &= \abs{\frac{1}{\varepsilon}\int_\Gamma \SG\left(W''(\varphi^N)\phit\right) \cdot \SG \mut} \nonumber \\ &\leq \frac{C_\gamma}{\varepsilon}\int_\Gamma \abs{\SG\left(W''(\varphi^N)\phit\right)}^2 + \frac{\gamma}{\varepsilon}\int_\Gamma \abs{\SG \mut}^2
		\end{align}
		where we have used that $\partial_t \phit = \SL \mut$ almost everywhere since by definition we have $\phit \in V_\G^N$ and $\mut \in V_\G^N$ for all $t \in (0,T),$ i.e. \eqref{eq:weak_dt_CH1_approx} implies for all $t \in (0,T)$ the identity $\partial_t \phit = \SL \mut$ in $V_\G^N$ and thus $\partial_t \phit = \SL \mut$ almost everywhere in $\G_T.$ 
		
		The first term on the right-hand side can be controlled by $\int_{\Gamma } \abs{\SG \phit}^2$ in the following way. By the growth properties of $W$, we have 
		\begin{align}\label{eq:higher_reg_W_est_step2}
		\int_\Gamma &\abs{\SG\left(W''(\varphi^N)\phit\right)}^2 \leq 2 \int_\Gamma \abs{W''(\varphi^N)}^2 \abs{\SG \phit}^2 + 2 \int_\Gamma \abs{\SG \left(W''(\varphi^N)\right)}^2 \abs{\phit}^2 \nonumber \\ &\leq C \left( \norm{\varphi^N(t)}^4_{L^\infty(\Gamma)} + 1 \right) \left( \int_{\Gamma} \abs{\SG\phit}^2 \right) + 2 \int_\Gamma \abs{\SG \left(W''(\varphi^N)\right)}^2 \abs{\phit}^2. 
		\end{align}
		Moreover, we apply Hölder's inequality to deduce
		\begin{align*}
		\int_\Gamma \abs{\SG\left(W''(\varphi^N)\right)}^2\abs{\phit}^2 &\leq C\int_{\Gamma} \abs{\varphi^N+1}^2\abs{\SG \varphi^N}^2\abs{\phit}^2 \nonumber
		\\
		&\leq C \left( \int_{\Gamma} \abs{\varphi^N+1}^6 \right)^{2/6} \left(\int_{\Gamma}\abs{\SG \varphi^N}^6\right)^{2/6} \left(\int_{\Gamma}\abs{\phit}^6\right)^{1/3}.
		\end{align*}
		Using that $\int_{\Gamma} \phit = \int_{\Gamma} \SL \mu^N = 0$ we have furthermore
		\[ \left(\int_{\Gamma}\abs{\phit}^6\right)^{1/3} \leq C \left(\int_{\Gamma}\abs{\SG\phit}^2\right)  \] 
		by the Sobolev embedding theorem. Hence
		\begin{align*}
		\int_\Gamma \abs{\SG\left(W''(\varphi^N)\right)}^2\abs{\phit}^2 & \leq C \left( \int_{\Gamma} \abs{\varphi^N+1}^6 \right)^{2/6} \left(\int_{\Gamma}\abs{\SG \varphi^N}^6\right)^{2/6} \left(\int_{\Gamma}\abs{\SG \phit}^2\right).
		\end{align*}
		Thus \eqref{eq:higher_reg_W_est_step2} reads
		\begin{align}\label{eq:higher_reg_W_est_step3}
		\int_\Gamma \abs{\SG\left(W''(\varphi^N)\phit\right)}^2 \leq& 2 \int_\Gamma \abs{W''(\varphi^N)}^2 \abs{\SG \phit}^2 + 2 \int_\Gamma \abs{\SG \left(W''(\varphi^N)\right)}^2 \abs{\phit}^2 \nonumber \\ \leq& C \left( \norm{\varphi^N(t)}^4_{L^\infty(\Gamma)} + 1 \right) \left( \int_{\Gamma} \abs{\SG\phit}^2 \right)\nonumber \\ &+ C \left( \int_{\Gamma} \abs{\varphi^N+1}^6 \right)^{2/6} \left(\int_{\Gamma}\abs{\SG \varphi^N}^6\right)^{2/6} \left(\int_{\Gamma}\abs{\SG \phit}^2\right).
		\end{align}
		We observe that $\left(\norm{\varphi^N(t)}^4_{L^\infty(\Gamma)}+1\right)$ is bounded in $L^1(0,T)$ by the following argument. Since $\varphi^N \in L^\infty(0,T;H^1(\G))$ and $\varphi^N \in L^2(0,T;H^2(\G)),$ Hölder's inequality implies $\varphi^N \in L^4\left(0,T;H^{3/2}(\G)\right),$ where $H^{3/2}(\G)$ is the interpolation space of exponent $s=1/2$ between $H^1(\G)$ and $H^2(\G).$ Hence the embedding \[\left(H^1(\G);H^2(\G) \right)_{1/2,2} = H^{3/2}(\G) \hookrightarrow C^{0,\alpha}(\G) \text{ for } 0<\alpha<1/2 \] yields $\varphi^N \in L^4(0,T;L^\infty(\G)).$ Likewise, \eqref{eq:int_phi_N} and \eqref{eq:int_SG_phi_N} imply 
		\[ \left( \int_{\Gamma} \abs{\varphi^N(t)+1}^6 \right)^{2/6} \in L^\infty(0,T) \quad \text{ and } \quad \left(\int_{\Gamma}\abs{\SG \varphi^N(t)}^6\right)^{2/6} \in L^1(0,T)\]
		uniformly in $N,$ from which we deduce that \[ \left(\int_{\Gamma} \abs{\varphi^N(t)+1}^6 \right)^{2/6}\left(\int_{\Gamma}\abs{\SG \varphi^N(t)}^6\right)^{2/6} \in L^1(0,T). \]
		Hence \[ F^N(t) := \max \left\lbrace \left( \int_{\Gamma} \abs{\varphi^N(t)+1}^6 \right)^{2/6} \left(\int_{\Gamma}\abs{\SG \varphi^N(t)}^6\right)^{2/6},  \left( \norm{\varphi^N(t)}^4_{L^\infty(\Gamma)} + 1 \right) \right\rbrace \in L^1(0,T)\]
		and there exists a constant $C>0$ such that \[ \norm{F^N}_{L^1(0,T)} \leq C \]
		uniformly in $N.$
	
		Combining \eqref{eq:higher_reg_W_est_step1} and \eqref{eq:higher_reg_W_est_step3} we arrive at
		\begin{align}\label{eq:higher_reg_est_W_final}
		\abs{\frac{1}{\varepsilon}\int_\Gamma W''(\varphi^N)\phit \partial_t \phit} \leq \frac{2C_\gamma}{\varepsilon} F^N(t) \left(\int_{\Gamma}\abs{\SG \phit}^2\right) + 2\frac{\gamma}{\varepsilon} \int_{\Gamma} \abs{\SG \mut}^2.
		\end{align}
		We have thus estimated the first term on the right-hand side in \eqref{eq:higher_reg_energy_eq} and it remains to control the second term on the right-hand side in this inequality. To this end, we compute
		\begin{align}\label{eq:higher_reg_dt_q_start}
		&\abs{\int_\Gamma \frac{d}{dt}\left(q(u^N,v^N)\right)\left(\thetat-\ut \right)} \nonumber \\ &\leq \int_{\Gamma} \abs{D_u q(u^N,v^N)}\abs{{\ut}}^2 + \int_\Gamma \abs{D_u q(u^N,v^N)}\abs{\ut}\abs{\thetat} \nonumber \\ &\quad + \int_\Gamma \abs{D_v q(u^N,v^N)}\abs{\vt}\abs{\ut} + \int_\Gamma \abs{D_v q(u^N,v^N)}\abs{\vt}\abs{\thetat} 
		\end{align}
		In order to shorten the estimate for the last three terms, let $f,g,h$ be measurable functions on $\G.$ We deduce for all $\gamma > 0$
		\begin{align}\label{eq:higher_reg_abstract_control_exchange}
		\int_{\Gamma} \abs{f}\abs{g}\abs{h} & \leq \norm{f}_{L^4(\Gamma)}\norm{g}_{L^2(\Gamma)} \norm{h}_{L^4(\Gamma)} \nonumber \\
		&\leq C_\gamma \norm{f}_{L^4(\Gamma)}^2\norm{g}^2_{L^2(\Gamma)} + \gamma \norm{h}_{L^4(\Gamma)}^2
		\end{align}
		from Young's inequality, where we used the generalized Hölder inequality in the first step.
		
		We remark that using the Sobolev Embedding Theorem and the Trace Theorem we can always estimate
		\[ \norm{\ut}_{L^4(\Gamma)} \leq C \norm{\ut}_{H^{1/2}(\Gamma)} \leq \tilde{C} \norm{\ut}_{H^1(B)}. \]
		Moreover, $\vt = \frac{\delta}{2} \thetat + \frac{1}{2}\phit$ and thus by Poincar\'{e}'s inequality
		\begin{equation}\label{eq:higher_reg_est_v_by_theta_varphi}
		\norm{\vt}_{L^2(\Gamma)} \leq \frac{\delta}{2}\norm{\thetat}_{L^2(\Gamma)}  + \frac{1}{2}\norm{\phit}_{L^2(\Gamma)} \leq \frac{\delta}{2}\norm{\thetat}_{L^2(\Gamma)}  + \frac{C}{2}\norm{\SG \phit}_{L^2(\Gamma)}. 
		\end{equation} 
		Choosing $f=D_v q(u^N,v^N), g=\vt, h=\ut$ and $f=D_u q(u^N,v^N), g= \thetat, h=\ut $ respectively in \eqref{eq:higher_reg_abstract_control_exchange}, we deduce
		\begin{align}\label{eq:higher_reg_dt_q1}
		\int_\Gamma & \abs{D_u q(u^N,v^N)}\abs{\ut}\abs{\thetat} + \int_\Gamma \abs{D_v q(u^N,v^N)}\abs{\vt}\abs{\ut} \nonumber \\ \leq & C_\gamma\left( \norm{D_u q(u^N,v^N)}^2_{L^4(\Gamma)} + \frac{\delta}{2}\norm{D_v q(u^N,v^N)}^2_{L^4(\Gamma)} \right) \norm{\thetat}^2_{L^2(\Gamma)} \nonumber \\&+ \frac{C_\gamma}{2}\norm{D_v q(u^N,v^N)}^2_{L^4(\Gamma)}\norm{\SG \phit}^2_{L^2(\Gamma)} + \gamma C \left( \norm{\ut}^2_{L^2(B)} + \norm{\nabla \ut}^2_{L^2(B)} \right).
		\end{align}
		Note that we used \eqref{eq:higher_reg_est_v_by_theta_varphi} to estimate $\norm{\vt}_{L^2(\Gamma)}.$
		
		Now we choose $f=D_v q(u^N,v^N), g=\vt, h=\theta^N$ in \eqref{eq:higher_reg_abstract_control_exchange} to obtain
		\begin{align}\label{eq:higher_reg_dt_q2}
		\int_\Gamma \abs{D_v q(u^N,v^N)}\abs{\vt}\abs{\thetat} \leq C_\gamma \norm{D_v q(u^N,v^N)}^2_{L^4(\Gamma)}&\left( \frac{\delta}{2}\norm{\thetat}^2_{L^2(\Gamma)}  + \frac{1}{2}\norm{\SG \phit}^2_{L^2(\Gamma)} \right) \nonumber \\ &+ \gamma C \left( \norm{\thetat}^2_{L^2(\Gamma)} + \norm{\SG \thetat}^2_{L^2(\Gamma)} \right).
		\end{align}
		Finally, we use again the trace and Sobolev embedding theorems together with the interpolation inequality $\norm{f}_{H^s(\G)} \leq C \norm{f}_{L^2(\G)}^{1-s}\norm{f}^s_{H^1(\G)}$ to estimate
		\begin{align}\label{eq:higher_reg_dt_q3}
		\int_{\Gamma} \abs{D_u q(u^N,v^N)}\abs{{\ut}}^2 \leq& \norm{D_u q(u^N,v^N)}_{L^3(\Gamma)}\norm{\ut}^2_{L^3(\Gamma)} \nonumber \\ \leq& C \norm{D_u q(u^N,v^N)}_{L^3(\Gamma)}\norm{\ut}^2_{H^{1/3}(\Gamma)} \nonumber\\ \leq& C \norm{D_u q(u^N,v^N)}_{L^3(\Gamma)}\norm{\ut}^2_{H^{5/6}(B)} \nonumber\\ \leq&
		C\norm{D_u q(u^N,v^N)}_{L^3(\Gamma)}\norm{\ut}^{1/3}_{L^{2}(B)}\norm{\ut}^{5/3}_{H^{1}(B)} \nonumber\\ \leq& C_\gamma \norm{D_u q(u^N,v^N)}^6_{L^3(\Gamma)}\norm{\ut}^{2}_{L^{2}(B)} + \gamma \norm{\ut}^{2}_{H^{1}(B)}.
		\end{align}
		To simplify the notation, we introduce
		\[ M^N(t) = \max \left\lbrace \norm{D_u q(u^N,v^N)}^6_{L^3(\Gamma)} , \norm{D_u q(u^N,v^N)}^2_{L^4(\Gamma)}, \norm{D_v q(u^N,v^N)}^2_{L^4(\Gamma)} \right\rbrace. \]
		The functions $u^N$ and $v^N$ fulfil the assumptions of Lemma \ref{l:higher_reg_growth_to_int} by \eqref{eq:energy_est}. Hence \eqref{eq:higher_reg_assumption1} implies $M^N(t) \in L^1(0,T).$ Moreover, the bound on $M^N$ in $L^1(0,T)$ is uniform in $N$ since it is derived from the uniform estimate \eqref{eq:energy_est}. 
		
		We combine \eqref{eq:higher_reg_dt_q_start}, \eqref{eq:higher_reg_dt_q1}, \eqref{eq:higher_reg_dt_q2}, and \eqref{eq:higher_reg_dt_q3} and obtain
		\begin{align}\label{eq:higher_reg_dt_q_final}
		&\abs{\int_\Gamma \frac{d}{dt}\left(q(u^N,v^N)\right)\left(\thetat-\ut \right)} \nonumber \\ &\leq C_\gamma \left( M^N(t)+1\right)\norm{\ut}^2_{L^2(B)} + C_\gamma \left( (1+\delta)M^N(t) +1 \right) \norm{\thetat}^2_{L^2(\Gamma)} \nonumber \\ &\quad + C_\gamma M^N(t) \norm{\SG \phit}^2_{L^2(\Gamma)} + \gamma C \norm{\SG \thetat}^2_{L^2(\Gamma)} + \gamma C \norm{\nabla \ut}^2_{L^2(B)},
		\end{align}
		which controls the second term on the right-hand side in \eqref{eq:higher_reg_energy_eq}.
		We thus return to \eqref{eq:higher_reg_energy_eq} and use \eqref{eq:higher_reg_est_W_final} and \eqref{eq:higher_reg_dt_q_final} to deduce
		\begin{align*}
                  \frac{\varepsilon}{2} \frac{d}{dt}\int_{\Gamma} &\abs{\SG \phit}^2 + \frac{\delta}{8}\frac{d}{dt}\int_{\Gamma} \abs{\thetat}^2 + \int_{\Gamma} \abs{\SG \mut}^2 +\int_{\Gamma}\abs{\SG \thetat}^2 
                  + \frac{1}{2}\frac{d}{dt}\int_{B} \abs{\ut}^2 + D \int_B \abs{\nabla \ut}^2 \nonumber \\ & \leq C_\gamma \left( M^N(t)+1\right)\norm{\ut}^2_{L^2(B)} + C_\gamma \left( (1+\delta)M^N(t) +1 \right) \norm{\thetat}^2_{L^2(\Gamma)} \nonumber \\ &\quad + C_\gamma \left(M^N(t)+\frac{2F^N(t)}{\varepsilon}\right) \norm{\SG \phit}^2_{L^2(\Gamma)} + \gamma C \norm{\SG \thetat}^2_{L^2(\Gamma)} \nonumber \\ &\quad + \gamma C \norm{\nabla \ut}^2_{L^2(B)} + 2\frac{\gamma}{\varepsilon} \int_{\Gamma} \abs{\SG \mut}^2.
		\end{align*}
		By taking $\gamma$ to be sufficiently small, we can absorb the gradient terms on the right-hand side and conclude
		\begin{align*}
                  \frac{\varepsilon}{2} \frac{d}{dt}\int_{\Gamma} &\abs{\SG \phit}^2 + \frac{\delta}{8}\frac{d}{dt}\int_{\Gamma} \abs{\thetat}^2 + \int_{\Gamma} \abs{\SG \mut}^2 +\int_{\Gamma}\abs{\SG \thetat}^2 
                                                                                                                                                                                                                                   + \frac{1}{2}\frac{d}{dt}\int_{B} \abs{\ut}^2 + D \int_B \abs{\nabla \ut}^2 \nonumber \\ & \leq C_\gamma \left( M^N(t)+1\right)\norm{\ut}^2_{L^2(B)} + C_\gamma \left( (1+\delta)M^N(t) +1 \right) \norm{\thetat}^2_{L^2(\Gamma)} \nonumber \\ &\quad + C_\gamma \left(M^N(t)+\frac{2F^N(t)}{\varepsilon}\right) \norm{\SG \phit}^2_{L^2(\Gamma)}.
		\end{align*}
		Because of $M^N(t) \in L^1(0,T)$ and $F^N(t) \in L^1(0,T)$ uniformly in $N,$ Gronwall's inequality yields \eqref{eq:energy_est_dt}.
	\end{proof}
	\textit{Second step: Taking the limit $N \rightarrow \infty$.} Estimate \eqref{eq:energy_est_dt} is uniform in $N$ and allows to extract weakly converging subsequences, which for convenience we denote again by $\ut,\phit, \mut$ and $\thetat.$ Hence there exist functions 
	\begin{align*}
	&\tilde{u} \in L^\infty(0,T;L^2(B))\cap L^2(0,T;H^1(B)),\\
	&\tilde{\varphi} \in L^\infty(0,T;H^1(\Gamma)), \\
	&\tilde{\theta} \in L^\infty(0,T;L^2(\Gamma))\cap L^2(0,T;H^1(\Gamma))  \text{ and }\\
	&\tilde{\mu} \in L^2(0,T;H^1(\Gamma))
	\end{align*}
	such that
	\begin{align}
	&\ut \rightharpoonup \tilde{u} \text{ in } L^2(0,T;H^1(B)),\label{eq:high_reg_weak_conv1}\\
	&\phit \rightharpoonup \tilde{\varphi} \text{ in } L^2(0,T;H^1(\Gamma)),\label{eq:high_reg_weak_conv2}\\
	&\thetat \rightharpoonup \tilde{\theta} \text{ in } L^2(0,T;H^1(\Gamma))  \text{ and } \label{eq:high_reg_weak_conv3}\\
	&\mut \rightharpoonup \tilde{\mu} \text{ in } L^2(0,T;H^1(\Gamma))\label{eq:high_reg_weak_conv4}.
	\end{align}
	We remark that these convergences allow us to conclude
        \[
          \partial_t u = \tilde{u}, \ \partial_t \varphi = \tilde{\varphi}, \ \partial_t \theta = \tilde\theta, \text{ and } \partial_t \mu=\tilde\mu
        \]
        in the sense of distributions.
	\begin{lemma}\label{l:higher_reg_sol_lin_sys}
		The tuple $(\tilde{u},\tilde{\varphi}, \tilde{v}, \tilde{\mu},\tilde{\theta})$ is a weak solution to
		\begin{alignat}{2}
		\partial_t \tilde{u} &= D \Delta \tilde{u} &\qquad \text{in } B \times (0,T]&,\label{eq:dt_diffU}\\
		- D \nabla \tilde{u} \cdot \nu & = D_u q(u,v)\tilde{u}+D_v q(u,v)\tilde{v} &  
		\text{on } \Gamma \times (0,T]&,\label{eq:dt_flux}\\
		\pd_t \tilde{\varphi} &= \SL \tilde{\mu}
		&\text{on } \Gamma \times (0,T]&,\label{eq:dt_CH1}\\
		\tilde{\mu} &= - \eps \SL \tilde{\varphi} + \eps^{-1}W''(\varphi)\tilde{\varphi} - \frac{1}{2}\tilde{\theta}
		&\text{on } \Gamma \times (0,T]&,\label{eq:dt_CH2}\\
		\frac{\delta}{4}\pd_t \tilde{\theta} &= \SL \tilde{\theta} - \frac{1}{2} \SL \tilde{\mu} + D_u q(u,v)\tilde{u}+D_v q(u,v)\tilde{v} 
		&\text{on } \Gamma \times (0,T]&\label{eq:dt_v}\\
		\tilde{\theta} &= \frac2\delta(2\tilde{v} -
		\tilde{\varphi})&\text{on } \Gamma \times (0,T].& \label{eq:dt_theta}
		\end{alignat}
	\end{lemma}
	\begin{proof}[Proof of Lemma \ref{l:higher_reg_sol_lin_sys}]
		We first observe that \eqref{eq:weak_dt_diffU_approx} implies a bound on $\norm{\partial_t \ut}_{L^2(0,T;(H^{1}(B))')}$ in the following way. Let $\kappa \in L^2(0,T;H^1(B))$ and denote by $P_N^B$ the orthogonal projection in $H^1(B)$ onto $V^N_B.$ Then
		\begin{align}\label{eq:est_time_der_approx_higher_reg}
		&\abs{\int_0^T \langle\partial_t \ut,\kappa\rangle_{\left(H^1(B)\right)',H^1(B)}} \leq D\int_0^T \int_{B} \abs{\nabla \ut \cdot \nabla P_N^B\kappa} + \int_0^T \int_{\G} \abs{\frac{d}{dt}q(u^N,v^N)P_N^B\kappa} \nonumber \\
		\leq& D\norm{\nabla \ut}_{L^2(0,T;L^2(B))}\norm{\kappa}_{L^2(0,T;H^1(B))} + \int_0^T \int_{\G} \abs{D_u q(u^N,v^N) \ut P_N^B\kappa} \nonumber \\ & \quad \ \qquad \qquad \qquad \qquad \quad \qquad \qquad \qquad  +\int_0^T \int_{\G} \abs{D_v q(u^N,v^N) \vt P_N^B\kappa}.
		\end{align} 
		The first term is bounded by \eqref{eq:energy_est_dt} from Lemma \ref{l:higher_reg_energy_time_der}.
		The second term can be estimated by
		\begin{align*}
		&\int_0^T \int_{\G} \abs{D_u q(u^N,v^N) \ut P_N^B\kappa} \leq \int_0^T \left( \int_{\G} \abs{D_u q(u^N,v^N) \ut}^{4/3}  \right)^{3/4} \left( \int_{\G} \abs{P_N^B\kappa}^4 \right)^{1/4} \\ 
		\leq& \left(\int_0^T \left( \int_{\G} \abs{D_u q(u^N,v^N) \ut}^{4/3}  \right)^{3/2}\right)^{1/2} \left( \int_0^T \left( \int_{\G} \abs{P_N^B\kappa}^4 \right)^{1/2} \right)^{1/2} \\ 
		\leq& \norm{D_u q(u^N,v^N) \ut}_{L^2(0,T;L^{4/3}(\G))}\norm{P_N^B\kappa}_{L^2(0,T;L^4(\G))} \\
		\leq& C \norm{D_u q(u^N,v^N) \ut}_{L^2(0,T;L^{4/3}(\G))}\norm{\kappa}_{L^2(0,T;H^1(B))}. 
		\end{align*}
		Moreover, 
		\begin{align*}
		&\norm{D_u q(u^N,v^N) \ut}_{L^2(0,T;L^{4/3}(\G))}^2 = \int_0^T \left(\int_{\G} \abs{D_u q(u^N,v^N) \ut}^{4/3} \right)^{3/2} \\
		\leq& \int_0^T \left( \int_{\G} \abs{D_u q(u^N,v^N)}^3 \right)^{2/3} \left( \int_{\G} \abs{\ut}^{12/5}\right)^{5/6} \\
		\leq& \left( \int_0^T \left( \int_{\G} \abs{D_u q(u^N,v^N)}^3 \right)^{6/3} \right)^{2/6} \left( \int_0^T \left(\int_{\G} \abs{\ut}^{12/5}\right)^{15/12} \right)^{2/3} \\
		=& \norm{D_u q(u^N,v^N)}_{L^6(0,T;L^3(\G))}^2\norm{\ut}_{L^3(0,T;L^{12/5}(\G))}^2,
		\end{align*}
		where the first term is bounded by Lemma \ref{l:higher_reg_growth_to_int} and the second term is bounded because analogously to $u \in L^p(0,T;L^{\frac{4p}{3p-4}}(\G))$ in the proof of Lemma \ref{l:higher_reg_growth_to_int} we obtain $\ut \in L^p(0,T;L^{\frac{4p}{3p-4}}(\G))$ for $2 \leq p < 4$ and choosing $p=3$ yields that $\ut$ is bounded in $L^3(0,T;L^{12/5}(\G)).$ 
		
		The last term in \eqref{eq:est_time_der_approx_higher_reg} is bounded by the same arguments, with $D_u(q(u^N,v^N))$ replaced by $D_v(q(u^N,v^N))$ and $\ut$ replaced by $\vt.$
		
		Similarly, we find bounds on $\partial_t \thetat$ and $\partial_t \phit$ in $L^2(0,T;H^{-1}(\G))$ from \eqref{eq:weak_dt_v_approx} and \eqref{eq:weak_dt_CH1_approx} respectively. These also imply a bound on $\partial_t \vt \in L^2(0,T;H^{-1}(\G)).$ 
		
		These bounds on the time derivatives  allow us to deduce 
		\begin{align*}
		\partial_t \ut \rightharpoonup \partial_t \tilde{u} \text{ in } L^2(0,T;\left(H^1(B)\right)'),& \ 
		\partial_t \phit \rightharpoonup \partial_t \tilde{\varphi} \text{ in } L^2(0,T;H^{-1}(\Gamma)), \\ \text{ and }
		\partial_t \vt \rightharpoonup \partial_t \tilde{v}& \text{ in } L^2(0,T;H^{-1}(\Gamma)).
		\end{align*}
		If we recall the proof of Theorem \ref{thm:existence}, we also see that in addition we can infer  
		\begin{align*}
		\tr{\ut} \rightarrow \tr{\tilde{u}} \text{ in } L^2(0,T;L^2(\G)),& \ 
		\phit \rightarrow \tilde{\varphi} \text{ in }  L^2(0,T;L^2(\G)), \text{ and} \\
		\vt \rightarrow \tilde{v}& \text{ in }  L^2(0,T;L^2(\G)).
		\end{align*}
		In all these cases, the convergence also holds pointwise almost everywhere.
		
		The convergence of $D_u(q(u^N,v^N))$, $D_v(q(u^N,v^N))$, and $W''(\varphi^N)$ is again a consequence of the theory of Nemytskii operators.  We thus obtain 
		\begin{align*}
		D_u(q(u^N,v^N)) &\to_{N\to\infty} D_uq(u,v) \text{ in } L^2(0,T;L^2(\G)),\\
                  D_v(q(u^N,v^N)) &\to_{N\to\infty} D_uq(u,v) \text{ in } L^2(0,T;L^2(\G)),\text{ and}\\
                  	W''(\varphi^N) &\to_{N\to\infty} W''(\varphi) \text{ in } L^2(0,T;L^2(\G)).
		\end{align*}
		Together with the foregoing results on the convergence of $\lbrace \ut \rbrace_{N\in\N}$ and $\lbrace \vt \rbrace_{N\in\N}$ this is sufficient to take the limit in the equations \eqref{eq:weak_dt_diffU_approx}, \eqref{eq:weak_dt_CH2_approx}, and \eqref{eq:weak_dt_v_approx}.
		
			
		The remaining terms in the equations \eqref{eq:weak_dt_diffU_approx}--\eqref{eq:weak_dt_v_approx} are linear in $(\ut,\phit,\vt,\mut,\thetat)$ which implies that the limit functions $(\tilde{u},\tilde{\varphi},\tilde{v},\tilde{\mu},\tilde{\theta})$ are weak solutions to \eqref{eq:dt_diffU}--\eqref{eq:dt_theta}, first for all test functions $\omega$ and $\kappa$ in $V_\Gamma^{N_0}$ and $V_B^{N_0}$  for some $N_0 \in \N$ respectively and by an analogous argument as at the end of the proof of Theorem \ref{thm:existence} subsequently also for all test functions $\omega \in H^1(0,T;H^1(\Gamma))$ and $\kappa \in H^1(0,T;H^1(B))$. As such, $(\tilde{u},\tilde{\varphi},\tilde{v},\tilde{\mu},\tilde{\theta})$ are a weak solution to \eqref{eq:dt_diffU}--\eqref{eq:dt_theta}.
	\end{proof}	
	\textit{Third step: Higher regularity for the full system.} We would like to apply elliptic regularity theory to equation \eqref{eq:dt_CH2}. So far we have seen that $\varphi,\tilde{\varphi} \in L^\infty(0,T;H^1(\G)).$ As before, the Sobolev Embedding Theorem thus yields $\varphi,\tilde{\varphi} \in L^\infty(0,T;L^p(\G))$ for all $1\leq p < \infty.$  The term $W''(\varphi)\tilde{\varphi}$ on the right-hand side in \eqref{eq:dt_CH2} is an element of $L^2(0,T;L^2(\G))$ because of $\abs{W''(\varphi)}\leq C(1+\abs{\varphi}^2)$ and Hölder's inequality thus implies
	\[\norm{W''(\varphi)\tilde{\varphi}}_{L^2(0,T;L^2(\G))} \leq C (\norm{\varphi}_{L^\infty(0,T;H^1(\G))}^2+1) \norm{\tilde{\varphi}}_{L^\infty(0,T;H^1(\G))}.\] Hence the right-hand side in \eqref{eq:dt_CH2} is in $L^2(0,T;L^2(\G))$ and as a first step we deduce  
	\[ \tilde{\varphi} \in L^2(0,T;H^2(\Gamma)) \hookrightarrow L^2(0,T;W^{1,p}(\G)) \text{ for all } 1\leq p < \infty \]
	from elliptic theory. We can improve this result since actually $W''(\varphi)\tilde{\varphi} \in L^2(0,T;H^1(\G))$ by the following argument. The gradient of $W''(\varphi)\tilde{\varphi}$ can be estimated by
	\begin{align*}
	&\int_0^T \int_{\G} \abs{\SG\left(W''(\varphi)\tilde{\varphi}\right)}^2 \leq \int_0^T \int_{\G} \abs{W''(\varphi)\SG\tilde{\varphi}}^2 + \int_0^T \int_{\G} \abs{W'''(\varphi)\SG\varphi\tilde{\varphi}}^2 \\
	\leq& \int_0^T \left(\int_{\G} \abs{W''(\varphi)}^4\right)^{1/2} \left(\int_{\G}\abs{\SG \tilde{\varphi}}^4\right)^{1/2} + \int_0^T \left(\int_{\G}\abs{W'''(\varphi)}^8\right)^{1/4}\left(\int_{\G}\abs{\tilde{\varphi}}^8\right)^{1/4}\left(\int_{\G}\abs{\SG \varphi}^4\right)^{1/2}, 
	\end{align*}
	which implies
	\begin{align*}
	\norm{\SG\left(W''(\varphi)\tilde{\varphi}\right)}_{L^2(0,T;L^2(\G))} \leq &\norm{W''(\varphi)}_{L^\infty(0,T;L^4(\G))}\norm{\SG\tilde{\varphi}}_{L^2(0,T;L^4(\G))} \\ &+ \norm{W'''(\varphi)}_{L^\infty(0,T;L^8(\G))}\norm{\tilde{\varphi}}_{L^\infty(0,T;L^8(\G))}\norm{\SG \varphi}_{L^2(0,T;L^4(\G))}.
	\end{align*}
	The regularity of $\varphi$ in \eqref{eq:int_SG_phi} and of $\tilde{\varphi}$ above thus imply $W''(\varphi)\tilde{\varphi} \in L^2(0,T;H^1(\G))$ and in turn we deduce
	from elliptic theory applied to \eqref{eq:dt_CH2} 
	\[ \tilde{\varphi} \in L^2(0,T;H^3(\Gamma)). \]  
	Since 
	\[ \partial_t u = \tilde{u}, \ \partial_t \varphi = \tilde{\phi}, \ \partial_t \theta = \tilde\theta, \text{ and } \partial_t \mu=\tilde{\mu} \]
	in the sense of distributions, this implies 
	\begin{align}
	&\partial_t u \in L^\infty(0,T;L^2(B))\cap L^2(0,T;H^1(B)),\label{eq:higher_reg_dt_u}\\
	&\partial_t \varphi \in L^\infty(0,T;H^1(\Gamma))\cap L^2(0,T;H^3(\Gamma)) \text{ and }\label{eq:higher_reg_dt_phi}\\
	&\partial_t \theta \in L^\infty(0,T;L^2(\Gamma))\cap L^2(0,T;H^1(\Gamma)).\label{eq:higher_reg_dt_theta}
	\end{align}
	Hence we can derive 
	\[ \mu \in L^\infty(0,T;H^3(\Gamma))\cap L^2(0,T;H^5(\Gamma)) \] because $\varphi$ and $\mu$ are weak solutions to \eqref{eq:CH1}.  
	
	Recall that $u \in L^\infty(0,T;L^2(B))\cap L^2(0,T;H^1(B))$ is a weak solution to
	\begin{alignat*}{2}
	\partial_t u &= D \Delta u &\qquad \text{in } B \times (0,T]&,\\
	- D \nabla u \cdot \nu & = q(u,v) &  
	\text{on } \Gamma \times (0,T]&,
	\end{alignat*} 
	where by the growth condition on $q(u,v)$ one can directly prove that $q(u,v) \in L^2(0,T;L^2(\Gamma))$ and from \eqref{eq:higher_reg_dt_u} we also have $\partial_t u \in L^2(0,T;L^2(B)).$ 
	Considering the elliptic problem
	\begin{alignat*}{2}
	D \Delta u &= f &\qquad \text{in } B &,\\
	- D \nabla u \cdot \nu & = g &  
	\text{on } \Gamma &,
	\end{alignat*} 
	we infer from Amann \cite[Remark 9.5 (a)]{HA} that this problem admits a solution $u \in H^{1}(B)$ for any \[ (f,g) \in H^{-1}(B) \times H^{-1/2}(\G) \] if and only if $\int_B f + \int_{\Gamma} g = 0.$ We denote the corresponding continuous solution operator by \[T:H^{-1}(B) \times H^{-1/2}(\G) \rightarrow H^1(B).\] On the other hand, it follows from the same reference or alternatively from \cite[Theorem 4.18]{WM} that $T$ is also continuous as an operator 
	\begin{equation}\label{eq:higher_reg_bulk_sol_op}
	T:L^2(B)\times H^{1/2}(\Gamma) \rightarrow H^2(B). 
	\end{equation}
	This allows us to consider the interpolation spaces 
	\begin{align*}
	H^{-1/2}(B)&=\left(H^{-1}(B),L^2(B)\right)_{1/2,2}, \\ 
	L^2(\G)&=\left(H^{-1/2}(\Gamma),H^{1/2}(\Gamma)\right)_{1/2,2}, \text{ and } \\ H^{3/2}(B)&=\left(H^1(B),H^2(B)\right)_{1/2,2}
	\end{align*}
	to deduce from the properties of interpolation spaces that $T$ must also be continuous as an operator \[ T:H^{-1/2}(B)\times L^2(\Gamma) \rightarrow H^{3/2}(B). \] 
	Given that $q(u,v) \in L^2(0,T;L^2(\Gamma))$ and $\partial_t u \in L^2(0,T;L^2(B)),$ we deduce that
	\[ u \in L^2(0,T;H^{3/2}(B)). \]
	Together with \eqref{eq:higher_reg_dt_u} we infer $u|_{\G} \in H^1(0,T;H^{1/2}(\G))$ and in particular 
	\[ u|_{\G} \in L^\infty(0,T;H^{1/2}(\G)) \hookrightarrow L^\infty(0,T;L^4(\G)) \]
	because of the Sobolev embedding theorem. Using $v=\frac{\delta}{4} \theta + \frac{1}{2} \varphi,$  \eqref{eq:higher_reg_dt_phi}, and \eqref{eq:higher_reg_dt_theta} we derive the same property for $v.$ Since $D_u q(u,v)$ and $D_v q(u,v)$ grow at most linearly by \eqref{eq:higher_reg_growth_assumption1}, we thus have 
	\[ D_u q(u,v), D_v q(u,v) \in L^\infty(0,T;L^4(\G)). \]
	We use this information to derive that
	\begin{align*}
	\norm{D_u q(u,v)\SG u}^2_{L^2(0,T;L^{4/3}(\G))}=&\int_0^T \left( \int_{\G} \abs{D_u q(u,v)}^{4/3}\abs{\SG u}^{4/3} \right)^{3/2} \\
	\leq& \int_0^T \left(  \left( \int_{\G} \abs{D_u q(u,v)}^4 \right)^{1/2}\left( \int_{\G}\abs{\SG u}^2 \right) \right)\\
	\leq& C \norm{D_u q(u,v)}_{L^\infty(0,T;L^4(\G))}^2\left( \int_0^T \int_{\G}\abs{\SG u}^2 \right)
	\end{align*}
	and 
	\begin{align*}
	\norm{D_v q(u,v)\SG v}^2_{L^2(0,T;L^{4/3}(\G))} \leq C \norm{D_v q(u,v)}_{L^\infty(0,T;L^4(\G))}^2\left( \int_0^T \int_{\G}\abs{\SG v}^2 \right),
	\end{align*}
	from which we obtain that 
	\[ \SG \left(q(u,v)\right) = D_u q(u,v)\SG u + D_v q(u,v)\SG v \in L^2(0,T;L^{4/3}(\G)).\]
	We recall that $q(u,v) \in L^2(0,T;L^2(\G))$ and deduce
	\[ q(u,v) \in L^2(0,T;W^{1,4/3}(\G)) \hookrightarrow L^2(0,T;H^{1/2}(\G)) \]
	from the Sobolev embedding theorem. Thus the mapping properties in \eqref{eq:higher_reg_bulk_sol_op} actually yield
	\begin{equation}\label{eq:higher_reg_final_reg_u}
	u \in L^2(0,T;H^2(B)).
	\end{equation}
	We have already seen that $W'(\varphi) \in L^\infty(0,T;L^2(\G))$ as well as $\theta, \mu \in L^\infty(0,T;L^2(\G)).$ As $\varphi$ is a solution to \eqref{eq:CH2}, we thus deduce
	\begin{equation}\label{eq:higher_reg_reg_phi_last_step}
	\varphi \in L^\infty(0,T;H^2(\G)).
	\end{equation} 
	Moreover, $\partial_t v \in L^\infty(0,T;L^2(\G))\cap L^2(0,T;H^1(\G))$ by \eqref{eq:higher_reg_dt_phi}, \eqref{eq:higher_reg_dt_theta}, and \eqref{eq:theta}. Since $\theta$ solves \eqref{eq:v} and in addition $q(u,v) \in L^2(0,T:L^2(\G)),$ we deduce 
	\[ \theta \in L^2(0,T;H^2(\G)) \]
	from elliptic regularity theory. Thus 
	\[ v = \frac{\delta}{4}\theta + \frac{1}{2} \varphi \in L^2(0,T;H^2(\G)).\]	
	By \eqref{eq:higher_reg_final_reg_u} we have $u \in L^2(0,T;H^{3/2}(\G))$ and in particular $\SG u \in L^2(0,T;L^4(\G)).$
	We repeat the calculations from before to deduce
	\begin{align*}
	\norm{D_u q(u,v)\SG u}^2_{L^2(0,T;L^{2}(\G))}\leq& C \norm{D_u q(u,v)}_{L^\infty(0,T;L^4(\G))}\left( \int_0^T \int_{\G}\abs{\SG u}^4 \right)^{1/2},
	\end{align*}
	i.e. $D_u q(u,v)\SG u \in L^2(0,T;L^2(\G)).$ Furthermore, $v \in L^2(0,T;H^2(\G))$ yields $D_v q(u,v)\SG v \in L^2(0,T;L^2(\G))$ in a completely analogous manner. 
	
	As a direct consequence, we infer that in fact $q(u,v) \in L^2(0,T;H^1(\G)).$ Together with $\partial_t v \in L^\infty(0,T;L^2(\G))\cap L^2(0,T;H^1(\G))$ we turn again to elliptic regularity theory to deduce
	\[ \theta \in L^\infty(0,T;H^2(\Gamma)) \cap L^2(0,T;H^3(\Gamma)). \]
	We return to the regularity of $\varphi$ in \eqref{eq:higher_reg_reg_phi_last_step}.
	$H^2(\G)$ is a Banach algebra and hence every polynomial in $\varphi$ belongs to $L^\infty(0,T;H^2(\G)).$ In particular, this holds true for $W''(\varphi).$  Therefore, we can estimate
	\begin{align*}
	\int_0^T \norm{\SG \left(W'(\varphi)\right)}^2_{H^2(\G)} = \int_0^T \norm{ W''(\varphi)\SG\varphi}^2_{H^2(\G)} \leq \norm{W''(\varphi)}_{L^\infty(0,T;H^2)}^2 \int_0^T\norm{\SG\varphi}^2_{H^2(\G)}
	\end{align*}
	where we can use \eqref{eq:int_SG_phi} to control the last term. Hence \[ W'(\varphi) \in L^2(0,T;H^3(\Gamma)), \]
	and as a consequence
	\[ \varphi \in L^2(0,T;H^5(\Gamma)) \cap L^\infty(0,T;H^2(\Gamma)), \]
	which completes the proof of Theorem \ref{thm:higher_reg}. 	
\end{proof}

\section{Longtime Existence and Stationary States for the Reduced Model} In this section we prove Theorem \ref{thm:ex_stat_sol} and Proposition \ref{prop:longtime}. Both results are based on a suitable reformulation of the reduced model. This reformulation is also the starting point for the following proof of Theorem \ref{prop:conv_OK}.

The key observation is that the mass conservation properties \eqref{eq:mass_cons} hold also for the reduced model. Thus we can decouple the system into a set of evolution equations for the mean values and a set of evolution equations for the mean value free parts. 

Projecting each equation onto its mean value free part, we arrive at
\begin{alignat}{2}
\pd_t \varphi_\G  &= \SL \mu_\G 
&&\text{on } \Gamma \times (0,T], \label{eq:dt_phi_red_gen}\\
\mu_\G  &= - \eps \SL \varphi_\G  + \eps^{-1}P_\G W'(\varphi) - \frac{\theta_\G }{2}
\qquad&&\text{on } \Gamma \times (0,T],\label{eq:mu_red_gen} \\
\pd_t v_\G  &= \SL \theta_\G  + P_\G  q(u,v)
&&\text{on } \Gamma \times (0,T], \label{eq:dt_v_red_gen}\\
\theta_\G  &=  \frac2\delta(2v_\G  -
\varphi_\G )&&\text{on } \Gamma \times (0,T], \label{eq:theta_red_gen}
\end{alignat}
together with the equations 
\begin{alignat}{2}
\frac{d}{dt}\int_B u(t) &= -\int_{\Gamma} q(u,v) &&\text{ on } (0,T], \label{eq:dt_u_red_gen}\\
\frac{d}{dt} \int_{\Gamma} \varphi &= 0 &&\text{ on } (0,T], \label{eq:dt_mean_phi_ed_gen}\\
\int_{\Gamma} v &= M - \int_{B} u &&\text{ on } (0,T], \label{eq:mean_v_red_gen} \\
\int_{\Gamma} \mu &= \int_{\Gamma} \left( \eps^{-1}W'(\varphi) - \frac{\theta}{2} \right)\qquad \quad &&\text{ on } (0,T], \label{eq:mean_mu_red_gen} \\
\int_{\Gamma} \theta &= \frac{2}{\delta}\int_{\Gamma} \left[ 2v - 1 -\varphi \right] &&\text{ on } (0,T] \label{eq:mean_theta_red_gen}
\end{alignat}
for the mean values, where again $f_\G:=P_\G f$ and $P_\G$ denotes the projection on the mean value free parts, i.e. $P_\G f := f - \frac{1}{\abs{\G}}\int_{\G} f.$
\subsection{Stationary States for the Reduced Model}\label{sec:proof_stat}
The goal of this section is to prove the existence of stationary solutions to the reduced model, i.e. Theorem \ref{thm:ex_stat_sol}. We work with the reformulation \eqref{eq:dt_phi_red_gen}--\eqref{eq:mean_theta_red_gen} from above. Since we are concerned with the existence of stationary solutions, any time derivatives in \eqref{eq:dt_phi_red_gen}--\eqref{eq:mean_theta_red_gen} are set to zero.

We recall that by Condition \ref{cond:mean_value_fix}, Equations \eqref{eq:dt_u_red_gen} and \eqref{eq:mean_v_red_gen} already determine the mean values of the cholesterol concentrations $u$ and $v.$
\begin{proof}[Proof of Theorem \ref{thm:ex_stat_sol}]
	W.l.o.g we can assume that the mean value of $\varphi$ vanishes, i.e. $\frac{1}{\abs{\G}}\int_{\Gamma} \varphi = m =0.$ This is due to the fact that we can always consider $\overline{\varphi}=\varphi-m$ and work with the translated double-well potential $\overline{W}(s)=W(s+m).$
	
	We first consider the equations \eqref{eq:dt_phi_red_gen}--\eqref{eq:theta_red_gen} for the mean value free functions $\varphi_\G, v_\G, \theta_\G, \mu_\G.$ Note that these equations do not depend on the mean value $\int_{\Gamma} \mu.$ 
	
	In particular, equation \eqref{eq:dt_phi_red_gen} implies that $\mu_\G$ is constant. Since $\int_{\G}\mu_\G=0,$ we thus directly deduce $\mu_\G=0.$ As such, equations \eqref{eq:dt_phi_red_gen}--\eqref{eq:theta_red_gen} reduce to 
	\begin{alignat*}{2}
	0 &= - \eps \SL \varphi_\G + \eps^{-1}P_\G W'(\varphi) - \frac{\theta_\G}{2}
	& \ \text{ on } \Gamma &,\\
	0 &= \SL \theta_\G + P_\G  q(u,v)
	&\text{ on } \Gamma &,\\
	\theta_\G &=  \frac2\delta(2v_\G -
	\varphi_\G)&\text{ on } \Gamma &.
	\end{alignat*}
	To begin with, recall that $W'(\varphi) = 4\varphi^3-4\varphi$ and that the projection $P_\G $ is linear.
	
	Let $Z$ denote the space
	\[ Z:=  H^1_{(0)}(\Gamma) \times H^1_{(0)}(\Gamma) \times H^1_{(0)}(\Gamma) \]
	and define for $\tau \geq 0$ by $T_\tau$ the solution operator which maps a given right-hand side $ (\tilde{\varphi},\tilde{v},\tilde{\theta}) \in Z$ onto the solution to the problem 
	\begin{alignat}{2}
	0 &= - \eps \SL \varphi_\G + 4\eps^{-1}P_\G (\left( \varphi_\G^3 - \tau\tilde{\varphi}\right)) - \frac{\theta_\G}{2}
	\qquad &&\text{on } \Gamma,\label{eq:mu_red_gen_stat_mf} \\
	0 &= \SL \theta_\G + \tau P_\G  q(S_B(\tilde{v}),\tilde{v}+S_\G(\tilde{v}))
	&&\text{on } \Gamma, \label{eq:dt_v_red_gen_stat_mf}\\
	\theta_\G &=  \frac2\delta(2v_\G -
	\varphi_\G)&&\text{on } \Gamma, \label{eq:theta_red_gen_stat_mf}
	\end{alignat}
	where $S_B$ and $S_\G$ are the operators provided by Condition \ref{cond:mean_value_fix}.
	Note that for $\tau = 0,$ the operator $T_0: Z \rightarrow Z$ maps every element of $Z$ onto the solution to
	\begin{alignat*}{2}
	0 &= - \eps \SL \varphi_\G + 4\eps^{-1}P_\G ( \varphi_\G^3 ) - \frac{\theta_\G}{2}\qquad
	&&\text{on } \Gamma, \\
	0 &= \SL \theta_\G 
	&&\text{on } \Gamma, \\
	\theta_\G &=  \frac2\delta(2v_\G -
	\varphi_\G)&&\text{on } \Gamma,
	\end{alignat*}
	which is zero, i.e $T_0$ is constant.
	\begin{lemma}\label{l:thm_stat_sol_l1}
		The operator $T_\tau: Z \rightarrow Z$ is well defined and compact.
	\end{lemma} 
	\begin{proof}[Proof of Lemma \ref{l:thm_stat_sol_l1}]
		Since $q$ has sublinear growth by assumption \eqref{eq:sublin_growth}, $\tilde{v} \in H^1_{(0)}(\Gamma)$ is given, and $S_B$ and $S_\G$ are continuous, we see that $\tau P_\G  q(S_B(\tilde{v}),\tilde{v}+S_\G(\tilde{v})) \in L^2(\Gamma).$ Equation \eqref{eq:dt_v_red_gen_stat_mf} has therefore a unique solution $\theta_\G \in H^2_{(0)}(\Gamma).$
		
		Let $V(s):=s^4$ be the convex part of $W.$ We now define $\mathcal{G}: L^2_{(0)}(\Gamma) \rightarrow \R \cup \lbrace +\infty \rbrace$ by
		\[ \mathcal{G}(\varphi_\G):=
		\begin{cases}
		\int_{\Gamma} \frac{\varepsilon}{2}\abs{\SG \varphi_\G}^2 + \frac{1}{\varepsilon}V(\varphi_\G) \qquad &\text{if } \varphi_\G \in H^1_{(0)}(\G), \\
		+\infty &\text{else.}
		\end{cases} \]
		Then $\mathcal{G}$ is a proper, convex, and lower semi-continuous functional by Fatou's lemma. By \cite[Example 2.3.4]{BRMOP}, its $L^2-$gradient $A: L^2_{(0)}(\Gamma) \supset D(A) \rightarrow L^2_{(0)}(\Gamma)$ is therefore a maximal monotone operator, given by
		\[ A\varphi_{\G} =  \left( -\varepsilon \SL \varphi_\G + \frac{1}{\varepsilon} P_\G  V'(\varphi_\G)\right)\qquad \text{for all }\varphi_\G\in \mathcal{D}(A)= H^2_{(0)}(\G). \] 
		Its domain is $D(A) = H^2_{(0)}(\G).$
		Moreover, for all $\varphi_\G \in D(A)$
		\[ \lim_{\norm{\varphi_\G}_{L^{2}(\Gamma)} \rightarrow \infty} \frac{\mathcal{G}(\varphi_\G)}{\norm{\varphi_\G}_{L^{2}(\Gamma)}} \geq \lim_{\norm{\varphi_\G}_{L^{2}(\Gamma)} \rightarrow \infty} C \frac{\norm{\varphi_\G}^{2}_{L^2(\G)}}{\norm{\varphi_\G}_{L^2(\Gamma)}} = +\infty \]
		and by Proposition 2.14 in \cite{BRMOP} we find that for every $f \in L^2_{(0)}(\Gamma)$ there exists a $\varphi_\G \in D(A)$ which solves \[ A\varphi_\G = f.\] The solution $\varphi_\G$ is unique since $A$ is strictly monotone. Indeed, already the $L^{2}-$gradient of $\int_{\Gamma} \frac{\varepsilon}{2}\abs{\SG \varphi_\G}^2$ is strictly monotone and $\int_{\Gamma} V(\varphi_\G)$ is convex itself.  
		Choosing $f = \left( \tau\tilde{\varphi}+\frac{\theta_\G}{2} \right)$ we have $f \in L^2_{(0)}(\Gamma)$ since $\theta_\G \in H^2_{(0)}(\Gamma)$ and $\tilde{\varphi}_\G \in H^1_{(0)}(\Gamma).$ Consequently, there exists a unique $\varphi_\G \in D(A) = H^2_{(0)}(\Gamma)$ which solves \[ A\varphi_\G = \left( -\varepsilon \SL \varphi_\G + \frac{1}{\varepsilon} P_\G  V'(\varphi_\G)\right) = \left( \tau\tilde{\varphi}+\frac{\theta_\G}{2} \right), \] i.e. equation \eqref{eq:mu_red_gen_stat_mf}.
		
		To conclude the proof, we note that \[ H^2_{(0)}(\Gamma)\times H^2_{(0)}(\Gamma) \times H^2_{(0)}(\Gamma) \]embeds compactly into $Z.$ Hence $T: Z \rightarrow Z$ is indeed compact.
	\end{proof} 
	
	The proof of Theorem \ref{thm:ex_stat_sol} is now based on a fixed point argument for $T_1.$ By the Leray-Schauder theorem, we have a solution for the fixed point equation
	\[ T_1 \vectthree{\varphi_\G}{v_\G}{\theta_\G} = \vectthree{\varphi_\G}{v_\G}{\theta_\G} \]
	if we can prove uniform a priori estimates for solutions to
	\begin{equation}\label{eq:hom_eq}
	T_\tau\vectthree{\varphi_\G}{v_\G}{\theta_\G} = \vectthree{\varphi_\G}{v_\G}{\theta_\G},
	\end{equation} 
	where $\tau \in (0,1).$
	\begin{lemma}\label{l:thm_stat_sol_l2}
		Let $\tau \in (0,1)$ and let $(\varphi_\G, v_\G, \theta_\G)$ be a solution to \eqref{eq:hom_eq}. 
		Then 
		\begin{equation}\label{eq:a_priori_fp1}
		\int_{\Gamma} \abs{\SG \varphi_\G}^2 + \int_{\Gamma} \varphi_\G^4 + \int_{\Gamma} \abs{\SG \theta_\G}^2 \leq C(\varepsilon, \Gamma, u)
		\end{equation}
		and
		\begin{equation}\label{eq:a_priori_fp2}
		\int_{\Gamma} \abs{\SG v_\G}^2 \leq C(\varepsilon, \Gamma, u, \delta)
		\end{equation}
		Both estimates are uniform in $\tau.$
	\end{lemma}	
	\begin{proof}[Proof of Lemma \ref{l:thm_stat_sol_l2}]
		We multiply equation \eqref{eq:mu_red_gen_stat_mf} by $\varphi_\G$ and equation \eqref{eq:dt_v_red_gen_stat_mf} by $\theta_\G.$ Taking the sum of both equations and integrating over $\Gamma$ yields
		\begin{align*}
		\int_{\Gamma} \abs{\SG \theta_\G}^2 + &\varepsilon \int_{\Gamma} \abs{\SG \varphi_\G}^2 + \frac{4}{\varepsilon} \int_{\Gamma}\varphi_\G^4  \\ &=  - \frac{1}{2}\int_{\Gamma} \theta_\G\varphi_\G + \tau \int_{\Gamma} \varphi_\G^2 + \tau \int_{\Gamma} q(u,v_\G)\theta_\G.	
		\end{align*}
		We use that $\tau \in (0,1)$ and Young's inequality for $\eta,\gamma > 0$ to deduce
		\begin{align*}
		\int_{\Gamma} \abs{\SG \theta_\G}^2 + &\varepsilon \int_{\Gamma} \abs{\SG \varphi_\G}^2 + \frac{4}{\varepsilon} \int_{\Gamma}\varphi_\G^4 \\ \leq&  \left(C(\eta) + 1 \right)\int_{\Gamma} \varphi_\G^2 + \eta \int_{\Gamma} \theta_\G^2 + \abs{\int_{\Gamma} q(u,v_\G)\theta_\G} \\
		\leq&  \left(C(\eta) + 1 \right)\gamma\int_{\Gamma} \varphi_\G^4 + \eta C \int_{\Gamma} \abs{\SG \theta_\G}^2 + C(\eta,\gamma,\Gamma) + \abs{\int_{\Gamma} q(u,v_\G)\theta_\G}.
		\end{align*}
		For $\eta$ sufficiently small, this inequality implies
		\begin{align*}
		\int_{\Gamma} \abs{\SG \theta_\G}^2 + &\varepsilon \int_{\Gamma} \abs{\SG \varphi_\G}^2 + \frac{4}{\varepsilon} \int_{\Gamma}\varphi_\G^4 \\ \leq& \left(C(\eta) + 1 \right)\gamma\int_{\Gamma} \varphi_\G^4 + C(\eta,\gamma,\Gamma) + \abs{\int_{\Gamma} q(u,v_\G)\theta_\G}.
		\end{align*}
		Subsequently, we choose $\gamma$ sufficiently small to infer
		\begin{align*}
		\frac{1}{2}\int_{\Gamma} \abs{\SG \theta_\G}^2 + \varepsilon \int_{\Gamma} \abs{\SG \varphi_\G}^2 + \frac{2}{\varepsilon} \int_{\Gamma}\varphi_\G^4 \leq C(\eta,\gamma,\Gamma) + \abs{\int_{\Gamma} q(u,v_\G)\theta_\G}.
		\end{align*}
		By the assumptions \eqref{eq:sublin_growth} on $q$ we can estimate the right-hand side in the above equation by
		\begin{align*}
		\abs{\int_{\Gamma} q(u,v_\G)\theta_\G} \leq C \int_{\Gamma} \abs{\theta_\G}\left( 1 + \abs{u}^{1/\alpha} + \abs{v_\G}^{1/\alpha}\right).
		\end{align*} 
		From Young's inequality we deduce
		\begin{align*}
		\int_{\Gamma} \abs{\theta_\G} \abs{v_\G}^{1/\alpha} \leq C \int_{\Gamma} \abs{\theta_\G}^{\frac{\alpha+1}{\alpha}}+\int_\Gamma\abs{v_\G}^{\frac{\alpha+1}{\alpha}} 
		\leq C(\rho) + \rho \left( \int_\Gamma \abs{\theta_\G}^2+\int_\Gamma \abs{v_\G}^2\right)
		\end{align*}
		since $\frac{2\alpha}{\alpha+1} > 1 \Leftrightarrow \alpha > 1$ and using equation \eqref{eq:theta_red_gen_stat_mf} we obtain
		\begin{align*}
		\int_{\Gamma} \abs{v_\G}^2 \leq \frac{1}{2}\left( \int_\Gamma \abs{\frac{\delta}{2}\theta_\G}^2 + \int_\Gamma \abs{\varphi_\G}^2 \right).
		\end{align*}
		Since $u \in \R$ is a given constant, the estimates for the remaining terms are straightforward and Poincar\'{e}'s inequality yields
		\begin{align*}
		\int_{\Gamma} \abs{\SG \theta_\G}^2 + &\varepsilon \int_{\Gamma} \abs{\SG \varphi_\G}^2 + \frac{4}{\varepsilon} \int_{\Gamma}\varphi_\G^4 \\ \leq& C(\eta,\gamma,\Gamma,\rho) + \rho C \left( \int_\Gamma \abs{\frac{\delta}{2}\SG \theta_\G}^2 + \int_\Gamma \abs{\SG \varphi_\G}^2 \right).
		\end{align*}
		Choosing $\rho$ sufficiently small, we deduce the estimate \eqref{eq:a_priori_fp1}. Estimate \eqref{eq:a_priori_fp2} now follows directly from equation \eqref{eq:theta_red_gen_stat_mf}. 
	\end{proof}
	
	Based on these two lemmas, we now proceed with the proof of Theorem \ref{thm:ex_stat_sol}.
	
	Let $T_\tau: Z \rightarrow Z$ be as before. By Lemma \ref{l:thm_stat_sol_l1}, $T_\tau$ is compact. By Lemma \ref{l:thm_stat_sol_l2} and the Poincar\'{e} inequality we have uniform a priori estimates in $Z$ on all solutions to \eqref{eq:hom_eq}. The Leray-Schauder principle \cite[Theorem 6.A]{ZEI} (or for the Leray-Schauder mapping degree theory, \cite[Chapter 13]{ZEI}) hence guarantees the existence of a fixed point of $T_1.$ This proves the theorem.
\end{proof}
\subsection{Boundedness in Time}\label{sec:long_ex}
We now prove Proposition \ref{prop:longtime}, which is a corollary of the following lemma.
\begin{lemma}
	Assume that Condition \ref{as:growth_q_longtime} holds. Then there exist constants $C,c>0$ which do not depend on $t$ such that
	\begin{equation*}
	\frac{d}{dt}\F(v,\varphi) \leq C - c\F(v,\varphi)\qquad \text{for all }t\in (0,\infty). 
	\end{equation*}
\end{lemma}
\begin{proof}
	We calculate
	\begin{align}\label{eq:long_time_derivativeF}
	\frac{d}{dt} \F(v,\varphi) = -\int_{\Gamma} \abs{\SG \mu_\G}^2 - \int_{\Gamma} \abs{\SG \theta_\G}^2 + \int_{\Gamma} \theta q(u,v)\quad \text{for all }t\in (0,\infty).
	\end{align}
	The last term can be estimated by
	\begin{align}
	\abs{\int_{\Gamma} \theta q(u,v)} &\leq C\left( \int_{\Gamma} \abs{\theta} \abs{1+\abs{u}^{1/\alpha}+\abs{v}^{1/\alpha}} \right) \nonumber \\
	&\leq C \left( \int_{\Gamma} \abs{\theta} + \int_{\Gamma} \abs{\theta}\abs{u}^{1/\alpha} +\int_{\Gamma} \abs{\theta}\abs{v}^{1/\alpha}\right). \label{eq:long_time_q1}
	\end{align}
	As before we find by Young's inequality
	\begin{align*}
	\int_{\Gamma} \abs{\theta}\abs{v}^{1/\alpha} &\leq C\left( \int_{\Gamma} \abs{\theta}^{\frac{\alpha+1}{\alpha}} + \int_{\Gamma} \abs{v}^{\frac{\alpha+1}{\alpha}} \right).
	\end{align*}
	We note that $\frac{\alpha+1}{\alpha}>1$ and hence conclude by Jensen's inequality
	\begin{align*}
	\int_{\Gamma} \abs{\theta}\abs{v}^{1/\alpha} &\leq C \int_{\Gamma} \abs{\theta}^{\frac{\alpha+1}{\alpha}} + C \int_\Gamma \abs{\frac{\delta}{4} \theta + \frac{\varphi+1}{2}}^{\frac{\alpha+1}{\alpha}} \\
	&\leq \int_{\Gamma} \abs{\theta}^{\frac{\alpha+1}{\alpha}} + C(\alpha) \int_\Gamma \abs{\frac{\delta}{4} \theta}^{\frac{\alpha+1}{\alpha}} + C(\alpha) \int_\Gamma \abs{\frac{\varphi+1}{2}}^{\frac{\alpha+1}{\alpha}}
	\end{align*}
	where we have also used that $v=\frac{\delta}{4} \theta + \frac{\varphi+1}{2}.$
	Since $\abs{\Gamma}<\infty,$ Hölder's inequality yields
	\begin{align*}
	\int_{\Gamma} \abs{\theta}\abs{v}^{1/\alpha} &\leq C(\delta) \left( \int_{\Gamma} \abs{\theta}^2\right)^{\frac{\alpha+1}{2\alpha}} + C \left(\int_{\Gamma} \abs{\varphi+1}^2\right)^{\frac{\alpha+1}{2\alpha}} 
	\end{align*}
	If we take into account that $u(t)\in \R$ is uniformly bounded in $t$ by Condition \ref{as:growth_q_longtime} and use Hölder's inequality to estimate the remaining terms in \eqref{eq:long_time_q1}, we arrive at
	\begin{align*}
	\abs{\int_{\Gamma} \theta q(u,v)} &\leq C_1 \left(\int_{\Gamma} \abs{\theta}^2 \right)^{1/2} + C(\delta)\left( \int_{\Gamma} \abs{\theta}^2 \right)^{\frac{\alpha+1}{2\alpha}} + C_2 \left( \int_{\Gamma} \abs{\varphi+1}^2 \right)^{\frac{\alpha+1}{2\alpha}} +C_3 \\ &\leq C(\delta)\left( \left(\int_{\Gamma} \abs{\theta}^2 \right)^{1/2} + \left( \int_{\Gamma} \abs{\theta}^2 \right)^{\frac{\alpha+1}{2\alpha}} + \left( \int_{\Gamma} \abs{\varphi+1}^2 \right)^{\frac{\alpha+1}{2\alpha}} +1\right).
	\end{align*}
	We define $\beta := \max{\lbrace \frac{1}{2},\frac{\alpha+1}{2\alpha}\rbrace}<1$. Using $\abs{\Gamma}<\infty$ again and Hölder's inequality we arrive at
	\begin{align*}
	\abs{\int_{\Gamma} \theta q(u,v)} &\leq C(\delta) \left( \left( \int_{\Gamma} \abs{\theta}^2 \right)^\beta + \left( \int_{\Gamma} \abs{\varphi+1}^2 \right)^\beta +1\right),
	\end{align*}
	which implies
	\begin{align}\label{eq:long_time_q2}
	\abs{\int_{\Gamma} \theta q(u,v)} &\leq C(\delta) \F(v,\varphi)^\beta +C
	\end{align}
	since $\beta < 1.$ 
	If we multiply equation \eqref{eq:mu_red_gen} by $\varphi_\G=\varphi-\fint_\Gamma \varphi $ and integrate over $\Gamma$ we obtain
	\begin{align*}
	\int_{\Gamma} \mu\varphi_\G + \frac{1}{2}\int_{\Gamma} \theta \varphi_\G = \varepsilon \int_{\Gamma} \abs{\SG \varphi_\G}^2 + \frac{1}{\varepsilon}\int_{\Gamma} W'(\varphi)\varphi_\G.
	\end{align*}
	The left-hand side can be estimated by
	\begin{align*}
	\int_{\Gamma} \mu\varphi_\G + \frac{1}{2}\int_{\Gamma} \theta \varphi_\G &\leq \frac{\varepsilon}{2} \int_{\Gamma} \abs{\SG \varphi_\G}^2 + C\int_{\Gamma} \abs{\SG\mu}^2 + C\int_{\Gamma} \abs{\SG\theta}^2.
	\end{align*}
	The double-well potential $W$ fulfils $W'(s)(s-m) \geq c_0 W(s) - c_1$ for $c_0, c_1 > 0.$ Thus the right-hand side above can be estimated from below by
	\begin{align*}
	\varepsilon \int_{\Gamma} \abs{\SG \varphi_\G}^2 + \frac{1}{\varepsilon}\int_{\Gamma} W'(\varphi)\varphi_\G &\geq \left(  \int_{\Gamma} \varepsilon \abs{\SG \varphi_\G}^2 + \frac{c_0}{\varepsilon} \int_{\Gamma} W(\varphi) \right) - \tilde{c}.
	\end{align*}
	Both estimates imply
	\begin{align}\label{eq:pot_control_energy}
	-\int_{\Gamma} \abs{\SG\mu}^2 -\int_{\Gamma} \abs{\SG\theta}^2 &\leq -C\left(\int_{\Gamma} \varepsilon \abs{\SG \varphi_\G}^2 + \frac{1}{\varepsilon} \int_{\Gamma} W(\varphi) \right) + \tilde{c}.
	\end{align}
	Next we observe that \[ \abs{\int_{\G} \theta \ } \leq \abs{\G}^{1/2}\left( \int_{\G} \abs{\theta}^2\right)^{1/2} \leq \abs{\G} + \int_{\G} \abs{\theta}^2 \leq \frac{2}{\delta}\F(v,\varphi) + C(\G). \]
	Thus by Poincar\'{e}'s inequality \[-\int_{\Gamma} \abs{\theta}^2 \geq - C\left(\int_{\Gamma} \abs{\SG \theta_\G}^2 + \F(v,\varphi)\right)-C(\G)\]
	and consequently
	\begin{align}\label{eq:pot_control_energy2}
	-\int_{\Gamma} \abs{\SG\mu}^2 -\int_{\Gamma} \abs{\SG\theta}^2 &\leq -C\left(\int_{\Gamma} \varepsilon \abs{\SG \varphi_\G}^2 + \frac{1}{\varepsilon} \int_{\Gamma} W(\varphi) + \int_{\Gamma} \frac{\delta}{2} \abs{\theta}^2\right) + \tilde{c}.
	\end{align} 
	Using \eqref{eq:long_time_q2} and $\eqref{eq:pot_control_energy2},$ we deduce
	\begin{align*}
	\frac{d}{dt} \F(v,\varphi) \leq C(\delta) \F(v,\varphi)^\beta - C\F(v,\varphi) + \tilde{c} 
	\end{align*}
	from \eqref{eq:long_time_derivativeF}. Finally Young's inequality allows us to deduce
	\begin{equation*}
	\frac{d}{dt}\F(v,\varphi) \leq C - c\F(v,\varphi),
	\end{equation*}
	which finishes the proof.
\end{proof}
\begin{proof}[Proof of Proposition \ref{prop:longtime}]The proposition is a direct corollary of the foregoing lemma. 
\end{proof}

\section{Convergence to the Ohta-Kawasaki equations as $\delta \rightarrow 0$}\label{sec:OK}

We are now interested in the limit process $\delta \rightarrow 0$ for the reduced model in the special case $q(u,v)=c_1u(1-v)-c_2v.$ In the following we will show that, if we set $\sigma = \theta_\G - \frac{1}{2}\mu_\G$ and send $\delta$ to zero in \eqref{eq:dt_phi_red_gen}--\eqref{eq:theta_red_gen}, we arrive at the limit problem
\begin{align}
\partial_t \varphi_\G &= \SL \mu_\G, \label{eq:OK_var1}\\
\frac{5}{4}\mu_\G &=  -\varepsilon \SL \varphi_\G + \frac{1}{\varepsilon} P_\G W'(\varphi_\G) - \frac{1}{2} \sigma, \\
\SL \sigma &= \frac{c_1u(t)+c_2}{2} \varphi_\G, \\
\int_{\Gamma} \sigma &= 0, \label{eq:OK_var4}
\end{align} 
which is a variant of the well-known Ohta-Kawasaki system which arises in the modelling of diblock copolymers, see \cite{OK} and \cite{ChR}.

We emphasize that the system we recover in the limit $\delta \searrow 0$ for the reduced model differs slightly from this system and in particular includes the time dependent factor $\frac{c_1u(t)+c_2}{2}$ in the equation for $\sigma.$ The function $u$ is given as the solution to the ordinary differential equation \eqref{eq:ode_u} and due to Remark \ref{rem:boundedness_u_red_model}, $u$ is bounded for all times if \[ u(0) \in [0,\abs{B}^{-1}M]. \] 
In particular, $u(t) \rightarrow u_\infty$ for $t \rightarrow \infty,$ where $ u_\infty $ is the positive zero of the right-hand side in \eqref{eq:ode_u}.

The main purpose of this section is the proof of Theorem \ref{prop:conv_OK}. Before we begin, we investigate the reformulation of the reduced problem \eqref{eq:dt_phi_red_gen}--\eqref{eq:mean_theta_red_gen} in the non-equilibrium case $q(u,v)=c_1 u (1-v) - c_2v.$

\subsection{A Reformulation for the Reduced Model in the Case $q=c_1 u (1-v) - c_2v$}
The explicit form of the exchange term $q$ allows us to eliminate $v$ in the equations \eqref{eq:dt_phi_red_gen}--\eqref{eq:mean_theta_red_gen}. In particular, \eqref{eq:flux_red} reduces to an explicit ODE for $u$ which does not depend on $\int_{\Gamma} v.$  

Let $(u,v,\varphi,\mu,\theta)$ be a solution to the reduced model. 

We use equation \eqref{eq:mean_theta_red_gen} to calculate
\begin{align*}
\int_{\Gamma} q(u(t),v(x,t)) = c_1 u(t) - (c_1 u(t) + c_2)\int_{\Gamma}\left( \frac{\delta}{4} \theta + \frac{1+\varphi}{2}\right)
\end{align*}
where we have used that $v=\frac{\delta}{4}\theta + \frac{1+\varphi}{2}$ almost everywhere. This infers 
\begin{align*}
P_\G  q(u,v) = -(c_1 u(t) + c_2)\left[ \frac{\delta}{4}\left( \theta - \fint_\Gamma \theta \right) + \frac{1}{2}\left(\varphi - \fint_\Gamma \varphi \right) \right].
\end{align*}
Thus we can rewrite equation \eqref{eq:dt_v_red_gen} to read
\begin{equation*}
\frac{\delta}{4} \partial_t \theta_\G  = \SL \theta_\G  -\frac{1}{2}\SL\mu_\G -\frac{\delta\left( c_1 u(t) + c_2 \right)}{4}\theta_\G  - \frac{\left( c_1 u(t) + c_2 \right)}{2}\varphi_\G ,
\end{equation*}
effectively eliminating $v_\G $ from the equation. 

Moreover, equation \eqref{eq:mean_theta_red_gen} enables us to eliminate $\int_{\Gamma} v$ in the equations for the mean values.

Summing up our findings, we obtain the system 
\begin{alignat}{2}
\pd_t \varphi_\G  &= \SL \mu_\G 
&\quad &\text{on } \Gamma \times (0,T], \label{eq:phi_red_alt}\\
\mu_\G  &= - \eps \SL \varphi_\G  + \eps^{-1}P_\G W'(\varphi) - \frac{\theta_\G }{2}
&&\text{on } \Gamma \times (0,T], \label{eq:mu_red_alt}\\
\frac{\delta}{4} \partial_t \theta_\G  &= \SL \theta_\G  -\frac{1}{2}\SL\mu_\G -\frac{\delta\left( c_1 u(t) + c_2 \right)}{4}\theta_\G  - \frac{\left( c_1 u(t) + c_2 \right)}{2}\varphi_\G 
&&\text{on } \Gamma \times (0,T],\label{eq:theta_red_alt}
\end{alignat}
together with
\begin{alignat}{2}
\frac{d}{dt}\int_B u(t) &= -\frac{c_1}{\abs{B}}\left( \int_B u(t)  \right)^2 + \left( c_1\frac{M-\abs{\Gamma}}{\abs{B}} - c_2 \right)\int_B u(t) +c_2M &\quad &\text{ on } (0,T], \label{eq:u_reduced} \\
\frac{d}{dt} \int_{\Gamma} \varphi &= 0 &&\text{ on } (0,T], \\
\frac{\delta}{4}\int_{\Gamma} \theta &= M - \int_{B} u - \int_{\Gamma} \frac{1+\varphi}{2}&&\text{ on } (0,T], \\
\int_{\Gamma} \mu &= \int_{\Gamma} \left( \eps^{-1}W'(\varphi) - \frac{\theta}{2} \right) &&\text{ on } (0,T],
\end{alignat}
which is an equivalent formulation for the reduced problem \eqref{eq:flux_red}--\eqref{eq:theta_red}. We remark that \eqref{eq:u_reduced} follows directly from \eqref{eq:ode_u} if we replace the modified exchange term $\tilde{q}$ by $q.$

The proof of Theorem \ref{prop:conv_OK} relies on the following lemma.
\begin{lemma}\label{l:energy_est_q_delta}
	Let $(u,\varphi,v,\mu,\theta)$ be a weak solution to the reduced model \eqref{eq:flux_red}--\eqref{eq:theta_red}. Then the mean value free parts $(u_\G,\varphi_\G,\mu_\G,\theta_\G,v_\G)$ fulfil for all $T<\infty$
	\begin{align*}
	\sup_{t\in(0,T)} \left[ \int_\Gamma\frac{\varepsilon}{2} \abs{\SG \varphi_\G(t)}^2 \right. &+ \left.\int_\Gamma \frac{1}{\varepsilon}W(\varphi(t)) + \int_\Gamma\frac{\delta}{8} \theta_\G^2(t) \right] \\ &+ \norm{\mu_\G}^2_{L^2(0,T;H^1(\Gamma))} + \norm{\theta_\G}^2_{L^2(0,T;H^1(\Gamma))} \leq C(T,\varepsilon, c_2),
	\end{align*}
	where $C(T,\varepsilon, c_2)$ depends on the initial data but is independent of $\delta.$
\end{lemma}
\begin{proof}
	We multiply equation \eqref{eq:theta_red_alt} by $\theta_\G$ and integrate over $\Gamma$ to obtain
	\begin{align}
	\frac{\delta}{8} \frac{d}{dt}\norm{\theta_\G}_{L^2(\Gamma)}^2 =& - \int_{\Gamma} \abs{\SG \theta_\G}^2 - \frac{1}{2}\int_{\Gamma} \partial_t \varphi_\G \theta_\G \nonumber \\ &- \frac{\delta}{4}\left( c_1 u(t) + c_2 \right)\int_{\Gamma}\theta_\G^2 - \frac{\left( c_1 u(t) + c_2 \right)}{2} \int_{\Gamma} \varphi_\G \theta_\G. \label{eq:energy_est_q_delta_step1}
	\end{align} 
	Furthermore, multiplying the equation
	\[ \mu_\G = -\varepsilon \SL \varphi + \frac{1}{\varepsilon}P_\G W'(\varphi) - \frac{1}{2}\theta_\G \]
	by $\partial_t \varphi_\G$ and integrating over $\Gamma$ yields 
	\[ \frac{1}{2}\int_{\Gamma} \partial_t \varphi_\G \theta_\G = \int_{\Gamma} \abs{\SG \mu_\G}^2 + \frac{d}{dt} \int_{\Gamma}\left[  \frac{\varepsilon}{2} \abs{\SG \varphi}^2 + \frac{1}{\varepsilon}W(\varphi) \right]. \]
	Substituting this into \eqref{eq:energy_est_q_delta_step1} implies
	\begin{align}
	\frac{d}{dt} \int_{\Gamma} \left[ \frac{\varepsilon}{2} \abs{\SG \varphi}^2 + \frac{1}{\varepsilon}W(\varphi) + \frac{\delta}{8} \theta_\G^2 \right] = -&\int_{\Gamma} \abs{\SG \theta_\G}^2 - \int_{\Gamma} \abs{\SG \mu_\G}^2 \nonumber \\-& \frac{\delta}{4}(c_1u(t)+c_2)\int_{\Gamma} \theta_\G^2 - \frac{c_1u(t)+c_2}{2} \int_{\Gamma} \varphi_\G \theta_\G \label{eq:energy_est_q_delta_step2}
	\end{align}
	Since $\abs{u(t)}<C$ for all $t \in (0,\infty)$ and some $C>0$ we deduce from Young's inequality for $\beta > 0$
	\[ \abs{\frac{c_1u(t)+c_2}{2} \int_{\Gamma} \varphi_\G \theta_\G } \leq C \left( \frac{1}{2\beta} \int_{\Gamma} \varphi_\G^2 + \frac{\beta}{2} \int_{\Gamma} \theta_\G^2 \right). \]
	Hence Poincar\'{e}'s inequality implies
	\[ \abs{\frac{c_1u(t)+c_2}{2} \int_{\Gamma} \varphi_\G \theta_\G } \leq C \left( \frac{1}{2\beta} \int_{\Gamma} \varphi_\G^2 + \frac{\beta}{2} \int_{\Gamma} \abs{\SG\theta_\G}^2 \right). \]
	We choose $\beta$ sufficently small to assure $C(\beta):=1-\beta\frac{C}{2}>0$ Thus \eqref{eq:energy_est_q_delta_step2} leads to the inequality
	\begin{align*}
	&\frac{d}{dt} \int_{\Gamma} \left[ \frac{\varepsilon}{2} \abs{\SG \varphi_\G}^2 + \frac{1}{\varepsilon}W(\varphi) + \frac{\delta}{8} \theta_\G^2 \right] \\ &\leq -C(\beta)\int_{\Gamma} \abs{\SG \theta_\G}^2 - \int_{\Gamma} \abs{\SG \mu_\G}^2 \nonumber - \frac{\delta}{4}(c_1u(t)+c_2)\int_{\Gamma} \theta_\G^2 + \frac{C}{\beta}\int_{\Gamma} \varphi_\G^2 \\
	&\leq -C(\beta)\int_{\Gamma} \abs{\SG \theta_\G}^2 - \int_{\Gamma} \abs{\SG \mu_\G}^2 \nonumber - \frac{\delta}{4}(c_1u(t)+c_2)\int_{\Gamma} \theta_\G^2 + \rho \frac{C}{\beta \varepsilon}\int_{\Gamma} W(\varphi) + C(\rho,\varepsilon,\beta)
	\end{align*}
	where we have used Young's inequality with $\rho > 0$ in the second inequality.
	
	By \eqref{eq:pot_control_energy} we have \[ -C\left( \int_{\Gamma} \abs{\SG \theta_\G}^2 + \int_{\Gamma} \abs{\SG \mu_\G}^2 \right) \leq - \int_{\Gamma} \left[\frac{\varepsilon}{2} \abs{\SG \varphi_\G} + \frac{1}{\varepsilon} W(\varphi)\right] + C \] and for $\rho$ sufficiently small we thus find
	\begin{align}
	\frac{d}{dt} &\int_{\Gamma} \left[ \frac{\varepsilon}{2} \abs{\SG \varphi_\G}^2 + \frac{1}{\varepsilon}W(\varphi) + \frac{\delta}{8} \theta_\G^2 \right]  +\frac{C(\beta)}{2}\int_{\Gamma} \abs{\SG \theta_\G}^2 + \frac{1}{2}\int_{\Gamma} \abs{ \SG \mu_\G}^2 \nonumber \\ &\leq - C(\beta,\rho,\varepsilon,c_2) \left[ \int_{\Gamma} \frac{\varepsilon}{2} \abs{\SG \varphi}^2 + \frac{1}{\varepsilon}W(\varphi) + \frac{\delta}{8} \theta_\G^2 \right] + C(\rho,\varepsilon,\beta)\label{eq:energy_est_q_delta_step3}
	\end{align}
	
	We use the differential form of Gronwall's inequality (see e.g. \cite[Appendix B.2(j)]{LCE})  to deduce
	\begin{equation*}
	\sup_{t\in(0,\infty)} \int_{\Gamma} \left[ \frac{\varepsilon}{2} \abs{\SG \varphi}^2 + \frac{1}{\varepsilon}W(\varphi) + \frac{\delta}{8} \theta_\G^2 \right] \leq C
	\end{equation*}
	and therefore, after integrating \eqref{eq:energy_est_q_delta_step3} in time 
	
	\begin{align}
	\int_{\Gamma} \left[ \frac{\varepsilon}{2} \abs{\SG \varphi_\G(T)}^2 \right. &+ \left. \frac{1}{\varepsilon}W(\varphi(T)) + \frac{\delta}{8} \theta_\G^2(T) \right] \nonumber \\ &+\frac{C(\beta)}{2}\int_0^T \int_{\Gamma} \abs{\SG \theta_\G}^2 + \frac{1}{2}\int_0^T \int_{\Gamma} \abs{ \SG \mu_\G}^2 \leq C(T)\label{eq:energy_est_q_delta}
	\end{align}
	for all $T<\infty.$ This proves the assertion of the lemma.
\end{proof}

Based on this uniform estimate we can prove the main result of this section.

\begin{proof}[Proof of Proposition \ref{prop:conv_OK}]
	We first observe that the solution $u^{\delta_n}$ to equation \eqref{eq:u_reduced} is bounded for all times, see also Remark \ref{rem:boundedness_u_red_model}. Moreover, the equation is independent of $\delta$ and the bound is thus also uniform in $\delta.$ 
	
	By \eqref{eq:energy_est_q_delta} we deduce $\delta_n \norm{\theta_\G^{\delta_n}}_{L^2(0,T,L^2(\Gamma))} \leq C \delta_n \norm{\SG \theta^{\delta_n}_\G}_{L^2(0,T,L^2(\Gamma))} \leq C\delta_n,$ which yields for all $\Psi \in C_c^\infty(\Gamma_T)$ \[ \abs{\delta_n \int_{\Gamma_T} \theta^{\delta_n}_\G \partial_t \Psi} \leq \delta_n \norm{\theta^{\delta_n}_\G}_{L^2(0,T,L^2(\Gamma))} \norm{\partial_t \Psi}_{L^2(0,T,L^2(\Gamma))}, \]
	i.e $\delta_n \partial_t \theta^{\delta_n}_\G \rightarrow 0$ in the sense of distributions as $\delta_n\rightarrow 0.$ 
	At the same time, we can estimate $\norm{\delta_n \partial_t \theta^{\delta_n}_\G}_{L^2(0,T;H^{-1}(\Gamma))}$ uniformly in $\delta_n$ since by \eqref{eq:theta_red_alt} for all $\eta \in L^2(0,T;H^{1}(\Gamma))$ we find
	\begin{align*} \abs{\int_{0}^{T} \int_{\Gamma} \delta_n \partial_t \theta^{\delta_n}_\G \eta} \leq& \abs{\int_0^T \int_{\Gamma} \SG \theta^{\delta_n}_\G \cdot \SG \eta} + \abs{\int_0^T \int_{\Gamma}\frac{1}{2} \SG \mu^{\delta_n}_\G \cdot \SG \eta}  \\ & +\abs{\int_0^T \int_{\Gamma}\frac{\delta_n(c_1 u^{\delta_n}(t) + c_2)}{4} \theta^{\delta_n}_\G \eta} + \abs{\int_0^T \int_{\Gamma}\frac{c_1 u^{\delta_n}(t) + c_2}{2} \varphi^{\delta_n}_\G \eta},
	\end{align*}
	which implies
	\[ \abs{\int_{0}^{T} \int_{\Gamma} \delta_n \partial_t \theta^{\delta_n}_\G \eta} \leq C \norm{\eta}_{L^2(0,T;H^{1}(\Gamma))} \]
	by Lemma \ref{l:energy_est_q_delta} and the boundedness of $u^{\delta_n}(t).$ 
	
	In particular, $\delta_n \partial_t \theta^{\delta_n}_\G$ is bounded in $L^2(0,T;H^{-1}(\Gamma)).$ Since $L^2(0,T;H^{-1}(\Gamma))$ is a Hilbert space, it is reflexive. Hence there exists a weakly converging subsequence in $L^2(0,T;H^{-1}(\Gamma))$ and some function $\chi \in L^2(0,T;H^{-1}(\Gamma))$ such that $\delta_n \partial_t \theta^{\delta_n}_\G \rightharpoonup \chi$ in $L^2(0,T;H^{-1}(\Gamma))$ as $\delta_n \rightarrow 0.$ Since $\chi$ must coincide with the vanishing distributional limit we deduce $\chi \equiv 0.$ 
	
	Exploiting equation \eqref{eq:phi_red_alt}, we deduce similarly that $\partial_t \varphi^{\delta_n}_\G$ is bounded uniformly in $\delta_n$ in $L^2(0,T;H^{-1}(\G)).$ As such, there exists a weakly converging subsequence $\partial_t \varphi^{\delta_n}_\G \rightharpoonup \tilde{\varphi}_\G.$
	
	The bounds in Lemma \ref{l:energy_est_q_delta} also infer the weak convergence of the mean value free functions $\varphi^{\delta_n}_\G, \theta_\G^{\delta_n}$ and $\mu_\G^{\delta_n}$ in the reflexive space $L^2(0,T;H^1(\Gamma)).$ Again, this convergence is meant up to a subsequence.
	
	Calculating the distributional time derivative $\partial_t \varphi^{\delta_n}_\G$ in $\mathcal{D}'(0,T;H^{1}(\G))$ shows $\partial_t \varphi_\G = \tilde{\varphi}_\G,$ i.e (after the extraction of a subsequence) we have \[ \varphi_\G^{\delta_n} \rightharpoonup \varphi_\G \text{ in } L^2(0,T;H^1(\Gamma))\cap H^1(0,T;H^{-1}(\Gamma)).\]
	The Aubins-Lions Theorem and Sobolev embeddings yield $\varphi^{\delta_n}_\G \rightarrow \varphi_\G$ in $L^p(0,T;L^p(\G))$ for every $1\leq p<\infty$. Since $\abs{W'(s)} \leq C \left(\abs{s}^3 + 1\right)$ for all $s\in\R$, we obtain by the continuity of Nemytskii operators
        \begin{equation*}
          W'(\varphi^{\delta_n}_\G)\rightharpoonup W'(\varphi_\G) \text{ in } L^2((0,T)\times\G).
        \end{equation*}
	As a result we can pass the limits in the weak formulations of equations \eqref{eq:phi_red_alt}, \eqref{eq:mu_red_alt}, and \eqref{eq:theta_red_alt}. We obtain for all $\eta \in L^2(0,T;H^1_{(0)}(\Gamma))$
	\begin{align*}
	\int_0^T \left\langle \partial_t \varphi_\G,\eta \right\rangle &= -\int_0^T \int_{\Gamma} \SG \mu_\G \cdot \SG \eta \\
	\int_\G^T \int_{\Gamma} \left(\mu_\G + \frac{\theta_\G}{2}\right) \eta &=  \varepsilon \int_0^T \int_{\Gamma} \SG \varphi_\G \cdot \SG \eta + \frac{1}{\varepsilon} \int_0^T \int_{\Gamma} W'(\varphi_\G)\eta \\
	-\int_0^T \int_{\Gamma} \SG \theta_\G \cdot \SG \eta + \frac{1}{2}\int_0^T \int_{\Gamma} \SG \mu_\G \cdot \SG\eta &= \int_0^T \int_{\Gamma} \frac{c_1u(t)+c_2}{2} \varphi_\G \eta.
	\end{align*}
	We denote by $\sigma$ the auxiliary function $\sigma := \theta_\G-\frac{1}{2}\mu_\G$ and find that the limit functions $\varphi_\G, \theta_\G,$ and $\mu_\G$ are weak solutions to problem \eqref{eq:OK_var1}--\eqref{eq:OK_var4}. 
\end{proof}


\end{document}